\documentclass[hidelinks,onefignum,onetabnum]{siamart251216}

\usepackage{lipsum}
\usepackage{amsfonts}
\usepackage{amsmath}
\usepackage{graphicx}
\usepackage{epstopdf}
\usepackage{verbatim}
\usepackage{multirow}
\usepackage{algorithmic}
\usepackage{subcaption}
\usepackage{bm}
\usepackage{soul}
\usepackage{tikz}
\usetikzlibrary{calc}
\newcommand{\bn}{\bm{n}}
\newcommand{\bx}{\bm{x}}
\newcommand{\by}{\bm{y}}
\usepackage[shortlabels]{enumitem}
\crefname{assumptionsi}{Assumption}{Assumptions}
\Crefname{assumptionsi}{Assumption}{Assumptions}
\ifpdf
  \DeclareGraphicsExtensions{.eps,.pdf,.png,.jpg}
\else
  \DeclareGraphicsExtensions{.eps}
\fi
\usepackage{setspace}
\usepackage{placeins}

\newsiamremark{remark}{Remark}
\newsiamremark{hypothesis}{Hypothesis}
\crefname{hypothesis}{Hypothesis}{Hypotheses}
\newsiamthm{claim}{Claim}

\newlist{assumptions}{enumerate}{1}
\setlist[assumptions]{label=A\arabic*., ref=A\arabic*}

\headers{Discontinuous Galerkin method for VEM}{J. G. Calvo and C. Herrera}

\title{A two-level preconditioner for discontinuous virtual element methods with irregular subdomains\thanks{Submitted to the editors DATE.
}}

\author{Juan G. Calvo\thanks{Escuela de Matemática - CIMPA, Universidad de Costa Rica
  (\email{juan.calvo@ucr.ac.cr}).}
\and César Herrera\thanks{Department of Mathematics, Purdue University
  (\email{herre125@purdue.edu}).}
}

\usepackage{amsopn}

\ifpdf
\hypersetup{
  pdftitle={A two-level preconditioner for discontinuous virtual element methods with irregular subdomains},
  pdfauthor={J. G. Calvo, and C. Herrera}
}
\fi

\begin{document}

\maketitle

\begin{abstract}
We introduce and analyze a two-level overlapping additive Schwarz preconditioner for discontinuous Galerkin discretizations of virtual element spaces on general polygonal meshes in two dimensions, for a second-order elliptic problem whose coefficient is constant on each subdomain and may have large jumps across the interface. The discretization is a weighted symmetric interior penalty formulation: the averages carry the coefficient and the penalty is scaled by its harmonic mean across each edge, which keeps the method consistent and makes the constants of the analysis independent of the coefficient. The coarse space is obtained from a virtual, continuous partition of unity with a single degree of freedom per subdomain vertex, and is therefore applicable to the irregular subdomains produced by a graph partitioner; the discrete harmonic extension used to build it may be replaced by a cheaper projection onto polynomials of low degree. We prove that the condition number of the preconditioned system is bounded by $C(1+\log(H/h))(1+H/\delta)$, where $H$, $h$ and $\delta$ denote the subdomain diameter, the mesh size and the overlap, and where $C$ is independent of the number of subdomains and of the contrast of the coefficient, provided the latter is quasi-monotone. This is the same bound known for the conforming case and the analysis allows for hanging nodes. We also show, by an explicit example, that the quasi-monotonicity assumption cannot be dispensed with. Numerical experiments with structured triangular, hexagonal and Voronoi meshes, and with subdomains obtained both by uniform partitioning and by METIS, confirm the theory.
\end{abstract}



\section{Introduction}
In this paper we introduce a preconditioner suitable for linear systems that arise from the discretization of partial differential equations with discontinuous Galerkin methods (DGM) for virtual element methods (VEM). Given a polygonal domain $\Omega \subset \mathbb{R}^2$, 
we focus on the following model problem:
\begin{equation}\label{problemToSolve}
    \left\{
    \begin{aligned}
    -\nabla \cdot ( \rho \nabla u) &= f && \text{in } \Omega,\\
    u &= 0 && \text{on } \partial\Omega,\\
    \end{aligned}
    \right.
\end{equation}
where $\rho$ is a given real function. We assume that $\rho\in L^\infty (\Omega)$, $f \in L^2(\Omega)$ and that there exists a positive constant $\rho_0$ such that $\rho(x)\geq \rho_0$ for all $x\in\Omega$. The weak formulation for problem \cref{problemToSolve} is then: Find $u\in H_0^1(\Omega)$ such that 
\begin{equation} \label{eq:weakForm}
    a(u,v):=\int_\Omega \rho \nabla u\cdot \nabla v = \int_\Omega fv =: (f,v)_0\ \quad \forall v\in H_0^1(\Omega).
\end{equation}
It is well known that problem \cref{eq:weakForm} is well-posed; see, e.g., \cite{brenner}.

VEM generalize the finite element framework to polygonal meshes, accommodating non-convex elements, elements with an arbitrary number of edges, and hanging nodes, while avoiding the explicit construction of basis functions; see \cite{AHMAD2013376,basicVEM,basisVEM2}. DGM, on the other hand, dispense with the requirement of interelement continuity, enforcing it weakly through penalty terms defined on element interfaces; see \cite{dgmElliptic,MR2431403}. Their combination is therefore particularly appealing, as both approaches are designed to handle mesh configurations beyond the scope of standard conforming finite elements. The resulting discontinuous virtual element discretization of \cref{problemToSolve} was introduced by Brezzi and Marini \cite{brezzi} for the specific case $\rho=1$. The formulation considered in this work follows their construction, with one important modification: we employ weighted averages and scale the penalty parameter using the harmonic mean of the coefficient across each edge. This choice preserves consistency while ensuring that the constants appearing in the analysis remain independent of the coefficient, even when $\rho$ is discontinuous across subdomains and with high variations on its magnitude.

The linear systems resulting from these discretizations are ill-conditioned and the two-level overlapping additive Schwarz methods provide a well-established approach in this context. Their application to DGM discretizations was investigated in \cite{Feng-Karakashian,10.1007/s10915-017-0453-4}. In these works, the coarse space is constructed from a triangular coarse space and this requirement can be restrictive in practical settings. For instance, when subdomains are generated using graph partitioning algorithms such as METIS, they consist of unions of fine elements with potentially complex geometries. The virtual coarse space proposed in \cite{CalvoJ,CalvoVEMSchwarz} overcomes this limitation in the conforming setting by constructing coarse basis functions from degrees of freedom prescribed along subdomain boundaries, rather than relying on coarse triangles. A two-level method for DGM on triangular meshes was also considered in \cite{CalvoSolano} for irregular subdomains. 
The linear systems resulting from these discretizations are ill-conditioned and the two-level overlapping additive Schwarz methods provide a well-established approach in this context. Their application to DGM discretizations was investigated in \cite{Feng-Karakashian,10.1007/s10915-017-0453-4}. In these works, the coarse space is constructed from a triangular coarse space and this requirement can be restrictive in practical settings. For instance, when subdomains are generated using graph partitioning algorithms, they consist of unions of fine elements with potentially irregular geometries. The virtual coarse space proposed in \cite{CalvoJ,CalvoVEMSchwarz} overcomes this limitation in the conforming setting by constructing coarse basis functions from degrees of freedom prescribed along subdomain boundaries, rather than relying on coarse triangles. A two-level method for DGM on triangular meshes was also considered in \cite{CalvoSolano} for irregular subdomains. 

Our main contributions can be summarized as follows: (i) We develop a coarse space for discontinuous VEM that is suitable for irregular subdomains, derived from the virtual partition of unity introduced in \cite{CalvoJ} and approximated in $V_h$ using the projectors of the virtual element method. (ii) We adopt a continuous coarse space across interfaces with exactly one degree of freedom associated with each subdomain vertex, limiting the size of the coarse problem with similar results as discontinuous alternatives with coarse systems three to four times larger. (iii) We establish a condition number bound similar as the conforming case, with a constant that is independent of both the number of subdomains and the coefficient contrast, under the assumption that $\rho$ is quasi-monotone. This shows that the discontinuous discretization does not introduce an additional condition number dependence. (iv) Contrast robustness relies on $\rho$-weighted averaging in both the continuous approximation and the discrete formulation employed in the analysis. (v) We further verify that the quasi-monotonicity assumption is essential for a constrast-independent bound. (vi) The analysis applies to general polygonal meshes and, in particular, accommodates meshes with hanging nodes. (vii) We avoid discrete harmonic extensions by projecting virtual functions into polynomial spaces of modest dimension, saving computational time.

The paper is organized as follows. In \Cref{section:Notation}, we introduce some notation and the usual discontinuous discretization for problem \cref{problemToSolve}. The VEM and its discontinuous version are then described in \Cref{section:DGM+VEM}. The two level overlapping Schwarz method is then introduced in \Cref{section:2DOS}, where we state our main result for the condition number of the preconditioned system. The technical tools and the proof of the main result are included in \Cref{tech_tools}, and experimental results are shown in \Cref{section:experiments}. Finally, some conclusions and final remarks are reported in \Cref{section:conclusions}.

\section{Discontinuous Galerkin Methods} \label{section:Notation}
This section introduces the notation and assumptions adopted throughout the article and reviews the standard DGM; for a detailed introduction, we refer to \cite{dgmElliptic,MR2431403}.

\subsection{Notation and assumptions}
Given a bounded domain $\mathcal{O}\subset\mathbb{R}^2$, we adopt the standard notation for the Lebesgue space $L^2(\mathcal{O})$ with norm $\|\cdot\|_{L^2(\mathcal{O})}$, and Sobolev spaces $H^s(\mathcal{O})$, $s\in\mathbb{R}$, with norm $\|\cdot\|_{H^s(\mathcal{O})}$ and semi-norm $|\cdot|_{H^s(\mathcal{O})}$. Additionally, $(\cdot, \cdot)_{{\cal O}}$ denotes the usual inner product in $L^2({\cal O})$ and the space of polynomials of degree at most $k$ defined on ${\cal O}$ is denoted by $\mathcal{P}_k({\cal O})$.

Let $\{\mathcal{T}_h\}_h>0$ be a family of decompositions of $\Omega$ into polygonal elements. A member of this family is called a \textit{fine mesh}. For each element $E\in\mathcal{T}_h$, we denote by $d_E$ and $h_E$ its number of edges and its diameter. Referring to the sides of elements $E$, we define $\mathcal{E}^I$ and $\mathcal{E}^B$ to be the set of all interior and boundary edges, respectively. For each $e \in \mathcal{E}^I$, we denote by $E^+$ and $E^-$ the two elements such that $e \in \partial E^{+} \cap \partial E^{-}$. If $e \in \mathcal{E}^B$ then $e \subset \partial E^+ \cap \partial \Omega$. Let ${\cal E}_h := {\cal E}^I \cup {\cal E}^B$. For $e \in {\cal E}_h $, we set $\langle u, v \rangle_e := \displaystyle \int_e uvds$.

The set of nodes contains the vertices of all the elements $E \in \mathcal{T}_h$ in the interior of $\Omega$ and is represented by $S_\mathcal{N}$, while the set $S_{\mathcal{N}}^E$ includes the nodes of
$\mathcal{T}_h$ that are on the boundary of $E$ and the interior of $\Omega$. We note that the set $S_{\mathcal{N}}^E$ can contain hanging nodes, i.e., vertices of other polygons that lie on an edge of $E$. For every node $\boldsymbol{x} \in S_{\mathcal{N}}$, the set $\mathcal{T}_{\boldsymbol{x}, h}$ is the subset of ${\cal T}_{h}$ containing all the elements $E$ such that $\boldsymbol{x} \in \partial E$. We split the nodes as $S_{\mathcal{N}}=\mathcal{N}_v\cup\mathcal{N}_b$, where $\mathcal{N}_v$ denotes the interior nodes and $\mathcal{N}_b$ the nodes lying on $\partial\Omega$.

Furthermore, the domain $\Omega$ is partitioned into $N$ non-overlapping and simply connected subdomains $\{\Omega_i\}^N_{i=1}$ of diameter $H_i$ which are the union of elements of the decomposition $\mathcal{T}_h$. The partition formed by $\{\Omega_i\}^N_{i=1}$ is called the \textit{coarse mesh} and is denoted by $\mathcal{T}_H$. In a similar way, the set of nodes of $\mathcal{T}_H$ contains the vertices of all the subdomains $\Omega_i$ in the interior of $\Omega$ and is represented by $S_{\mathcal{N}}^H$, while
the set $S_{\mathcal{N}}^{\Omega_i}$ includes only the nodes of $\mathcal{T}_H$ that are on the boundary of $\Omega_i$ and not on $\partial \Omega$. Again, hanging nodes are allowed on the coarse mesh. We then construct overlapping subdomains $\Omega'_i \supset \Omega_i$ by adding layers of elements that are external to $\Omega_i$, and we denote by $\delta_i$ the maximum width of the region $\Omega'_i \backslash \Omega_i$. We also assume that $\rho$ is constant on each subdomain and write $\rho_i:=\rho|_{\Omega_i}$.

We also partition the set of nodes $S_\mathcal{N}$ into disjoint sets. The nodes that belong to exactly one subdomain $\Omega_i$ are its interior nodes. A subdomain edge is defined as the interior of the intersection of the closure of two neighboring subdomains $\Omega_i$ and $\Omega_j$. If such intersection has more than one component, each open component is considered as a subdomain edge. Then, the endpoint nodes of the subdomain edges belong to the set of subdomain vertices, which is denoted by $S_\mathcal{V}$. The interface of the decomposition includes the closure of all the subdomain edges.

We conclude this section by stating the assumptions that we require in this work. It is assumed that every partition $\mathcal{T}_h$ satisfies the following properties: 
\begin{assumptions}
    \item \label{ass:A1} Its elements are shape regular: there exists $\gamma>0$ such that every $E\in\mathcal{T}_h$ is star-shaped with respect to a ball of radius greater than or equal to $\gamma h_E$.
    \item \label{ass:A3} There exists $\gamma'>0$ such that, for every $E\in\mathcal{T}_h$, the distance between any two nodes of $S^E_{\mathcal{N}}$ is greater than or equal to $\gamma' h_E$. In particular $|e|\approx h_E$ for every edge $e\subset\partial E$ and $n_0:=\max_{E\in\mathcal{T}_h}d_E$ is finite and bounded in terms of $\gamma$ and $\gamma'$ alone.
    \item \label{ass:A2} It is locally quasi-uniform: if two elements $E_i$ and $E_j$ are adjacent, then $h_{E_i} \approx h_{E_j}$.
\end{assumptions}
Moreover, we have the following assumption for $\rho$:
\begin{assumptions}[start=4]
\item \label{ass:monotonoCoef} The coefficient $\rho$ is quasi-monotone in the sense of \cite{DryjaSarkisWidlund}, adapted to the present two-dimensional setting as follows: for every interior node $\bx\in S_\mathcal{N}$, let $E^{\max}_{\bx}\in \mathcal{T}_{\boldsymbol{x}, h}$ satisfy $\rho_{E^{\max}_{\bx}}=\max_{E\in \mathcal{T}_{\boldsymbol{x}, h}}\rho_E$. Then, every $E\in \mathcal{T}_{\boldsymbol{x}, h}$ can be joined to $E^{\max}_{\bx}$ by a path $E=F_1,F_2,\ldots,F_L=E^{\max}_{\bx}$ of elements of $\mathcal{T}_{\boldsymbol{x}, h}$ in which consecutive elements share an edge and $\rho_{F_1}\leq\rho_{F_2}\leq\cdots\leq\rho_{F_L}$. For every boundary node $\bx$ the same holds with $E^{\max}_{\bx}$ replaced by some element of $\mathcal{T}_{\boldsymbol{x}, h}$ having an edge on $\partial\Omega$. 
\end{assumptions}
\begin{remark}
    Quasi-monotonicity was introduced in \cite{DryjaSarkisWidlund} for substructures in three dimensions. The two usual variants, according to whether consecutive substructures along the path are required to share a face or only an edge, are stated as \cite[Assumptions 4.4 and 4.5]{DohrmannWidlundRGDSW}. This hypothesis is used in our analysis only in \cref{N_term_rho} to obtain a bound that is independent of the contrast.
\end{remark}

\subsection{Discontinuous Galerkin Methods} This section presents the standard DGM discretization for triangular meshes; see \cite[Chapter 2]{MR2431403} for additional details. The variational discontinuous formulation for problem \cref{problemToSolve} is defined in the broken Sobolev space \( H^s(\mathcal{T}_h) := \prod_{E \in \mathcal{T}_h} H^s(E) \). We regard each element \( E \) as an open set, so that functions in \( H^s(\mathcal{T}_h) \) are single-valued. For $v \in H^s(\mathcal{T}_h)$, $s \geq 1$ and $e \in \mathcal{E}^I$, $v^{+}_e$ and $v^{-}_e$ denote the trace on $e$ of the restriction of $v$ to $E^{+}$ and $E^{-}$, respectively. The unit normal vector oriented from $E^+$ to $E^-$ is denoted by $\bn_e$. We write $\rho^{\pm}=\rho|_{E^{\pm}}$. Define the weights $w_e^\pm$ and the harmonic mean $\rho_e$ of the coefficient across $e$ by
\begin{equation} \label{eq:harmonicMean}
    w_e^+=\frac{\rho^{-}}{\rho^++\rho^-},\quad w_e^-=\frac{\rho^{+}}{\rho^++\rho^-}, \quad \rho_e=\dfrac{2\rho^+\rho^-}{\rho^++\rho^-},
\end{equation}
with the convention $\rho_e=\rho^+$ for $e\in\mathcal{E}^B$. In addition, we define the jump and the weighted average of $v$ on $e$ as
\begin{align*}
    [v] = \left\{\begin{array}{ll}
     v_e^{+} - v^{-}_e, & e\in \mathcal{E}^I,\\[2mm]
    v_e^+, & e \in \mathcal{E}^B,
    \end{array}\right. \quad 
    \{v\}_w = \left\{\begin{array}{ll}
     w_e^{+}v_e^{+} + w_e^{-}v^{-}_e, & e\in \mathcal{E}^I,\\[2mm]
    v_e^+, & e \in \mathcal{E}^B.
    \end{array}\right.
\end{align*}
We remark that for $e\in\mathcal{E}^I$ it holds that
\begin{equation} \label{eq:propW}
    w_e^+\rho^+ = w_e^-\rho^- = \frac{\rho_e}{2},\quad \text{and }\quad \min\{\rho^+,\rho^-\}\leq \rho_e\leq 2\min\{\rho^+,\rho^-\}. 
\end{equation}
Therefore, we have that the weighted average of the flux is 
\begin{equation} \label{eq:averageWe}
    \{\rho\nabla v\cdot\bn_e\}_w = \dfrac{\rho_e}{2}\left(\nabla v^+\cdot\bn_e+\nabla v^-\cdot\bn_e\right).
\end{equation}
For $s>3/2$, consider the bilinear form $A_h(\cdot, \cdot): H^s(\mathcal{T}_h) \times H^s(\mathcal{T}_h) \rightarrow \mathbb{R}$ given by
\begin{align*}
A_h(u, v) &= a_h(u, v) + b_{h}(u, v), \text{ where } a_h(u, v) = \sum_{E \in \mathcal{T}_h}(\rho\nabla u, \nabla v)_{E},\\
b_h(u, v) &= -\sum_{e \in \mathcal{E}_h} \biggl(\langle\{\rho  \nabla u \cdot \bn_e\}_w, [v]\rangle_{e} + \langle \{\rho\nabla v \cdot \bn_e\}_w, [u]\rangle_{e}\biggr)+\sum_{e \in \mathcal{E}_h}\frac{\gamma_h \rho_e}{|e|}\langle [u],[v]\rangle_e ,
\end{align*}
and  $\gamma_h> 0 $ is a given stability constant. Consider the functional $F_h : H^s(\mathcal{T}_h) \rightarrow \mathbb{R}$ defined by $F_h(v) = \sum_{E \in {\cal T}_h}(f, v)_{E}$. 
Hence, the formulation arising from DGM for problem \cref{problemToSolve} is to find $u \in H^s(\mathcal{T}_h)$ such that
\begin{equation*}
A_h(u,v) = F_h(v) \quad \forall \: v \in H^s(\mathcal{T}_h);
\end{equation*}
we refer to \cite[Chapter 2]{MR2431403} for consistency, well-posedness, and error estimates for the case where the weights are $w_e^\pm=1/2$.

\begin{remark}
The formulation is consistent: if $u\in H^s(\mathcal{T}_h)\cap H^1_0(\Omega)$ solves \cref{problemToSolve}, then $[u]=0$ on every $e\in\mathcal{E}_h$ and, since $w_e^++w_e^-=1$ and the normal flux $\rho\nabla u\cdot\bn_e$ is continuous across $e$, $\{\rho\nabla u\cdot\bn_e\}_\omega=\rho\nabla u\cdot\bn_e$. Hence $A_h(u,v)=F_h(v)$ for all $v\in H^s(\mathcal{T}_h)$. 
\end{remark}

\begin{remark}
For the general case of mixed boundary conditions, the functional \( F_h \) changes, while the bilinear form \( A_h \) remains unchanged. More precisely, suppose the boundary \(\partial\Omega\) is partitioned into two disjoint subsets \(\Gamma_D\) and \(\Gamma_N\). We impose the conditions  
\[
u = g_D \text{ on } \Gamma_D, \quad \text{and} \quad \rho\nabla u\cdot\bn = g_N \text{ on } \Gamma_N,
\]
where \(g_D \in H^{1/2}(\Gamma_D)\) and \(g_N \in L^2(\Gamma_N)\).  
Let \(\mathcal{E}^D\) and \(\mathcal{E}^N\) denote the sets of boundary edges lying on \(\Gamma_D\) and \(\Gamma_N\), respectively. Then the functional \(F_h(v)\) takes the form  
\begin{equation*}
F_h(v) = \sum_{E \in \mathcal{T}_h} (f, v)_E  
+ \sum_{e \in \mathcal{E}^N} \langle g_N, v\rangle_e  
- \sum_{e \in \mathcal{E}^D} \left\langle g_D, \rho\nabla v \cdot\bn_e - \frac{\gamma_h \rho_e}{|e|}v\right\rangle_e.
\end{equation*}
\end{remark}

\section{Discontinuous virtual methods} \label{section:DGM+VEM}

This section presents a discontinuous discretization for problem \cref{problemToSolve}. The discretization utilizes the basic theory of VEM for general polygonal meshes \cite{AHMAD2013376, basicVEM, basisVEM2} and DGM \cite{dgmElliptic}; see \cite{brezzi} for an introduction to discontinuous VEM for the case $\rho=1$. 

\subsection{Virtual element methods}
Consider a polygonal mesh satisfying Assumptions \ref{ass:A1} and \ref{ass:A3}. Given a positive integer $k$, for each element  $E \in \mathcal{T}_h$, define the set $\mathcal{B}_k(\partial E) : = \{v \in C^0(\partial E) : v|_e \in \mathcal{P}_k(e) \quad \forall \:  e \subset \partial E\}$, where $e$ represents any edge on the boundary of $E$. The local virtual element space of order $k$ is then defined as 
\begin{equation*}\label{localVirtual}
    V^E_k :=\{v \in H^1(E) : v \in B_k(\partial E), \quad \Delta v \in {\cal P}_{k-2}(E)\}.
\end{equation*}
We note that we consider ${\cal P}_{-1}(D) = \{0\}$. Hence, a function in $V^E_1$ is continuous and piecewise linear on the boundary of $E$, harmonic in the interior, and it is completely determined by its values at the nodes in $S^E_\mathcal{N}$.

In general, it is not possible to compute $(\rho \nabla u, \nabla v)_E$ for $u, v \in V_k^{E}$. Therefore, it is required to define a projector onto a polynomial space. Define the operator $\Pi^{\nabla}_{E, k} : V^E_k \rightarrow \mathcal{P}_k(E) \subset V^E_{k}$, where $\tilde{v} := \Pi^{\nabla}_{E, k} v$ is the polynomial that satisfies 
\begin{equation*}
   \left\{\begin{array}{l} (\nabla \tilde{v}, \nabla q)_E = (\nabla v, \nabla q)_E \quad \forall \: q \in \mathcal{P}_k(E),\\P_0(\tilde{v}) = P_0(v).
   \end{array}
   \right.
\end{equation*}
Here, $P_0 : V_k^E \rightarrow \mathbb{R}$ is a projection operator onto constants, which is chosen as the nodal average if $k = 1$ or the mean value if $k \geq 2$. It follows from the definition of $\Pi^{\nabla}_{E, k}$ that
\begin{equation*} 
    a^E(u, v) = a^E(\Pi^{\nabla}_{E, k}u,\Pi^{\nabla}_{E, k}v) + a^E(u-\Pi^{\nabla}_{E, k}u,v-\Pi^{\nabla}_{E, k}v),
\end{equation*}
which justifies to choose the local bilinear form for VEM for functions $u, v \in V_h$ as 
\begin{align}\label{localBilinearForm}
    \hat{a}_h(u, v) :=& \sum_{E \in \mathcal{T}_h}\hat{a}^E_h(u, v) = \sum_{E \in \mathcal{T}_h} (\rho \nabla \Pi^\nabla_{E,k}u, \nabla \Pi^\nabla_{E,k}v)_E + S^{E}(u - \Pi^\nabla_{E,k}u, v - \Pi^\nabla_{E,k}v),
\end{align}
with $S^E$ an stabilizing term that satisfies 
\begin{equation} \label{eq:stabIneq}
    c_1 a^E(v,v)\leq S^E(v,v)\leq c_2 a^E(v,v)\quad \forall \:v\in V_k^E \: \text{ with } \: \Pi^\nabla_{E,k}v = 0,
\end{equation}
for some positive constants $c_1,c_2$ independent of $E$, $h_E$ and $\rho_E$. For the lowest order case, it is usually chosen as $S^E(u, v):=\rho_E \sum_{\bx \in S^E_{\mathcal{N}}}u(\bx)v(\bx)$. We refer to \cite{basicVEM,basisVEM2} for an introduction to VEM, well-posedness and error estimates of this problem.

\subsection{The virtual discontinuous space and discretization}Let $s>3/2$. For simplicity, we consider the global space of lowest-order discontinuous piecewise virtual element functions defined as $V_h := \{v \in H^{s}(\mathcal{T}_h) : v|_{E} \in V^E_1\quad\forall \: E \in \mathcal{T}_h\}$. 
Since no continuity is imposed between elements, $V_h$ is the product of the local spaces $V^E_1$ and a function of $V^E_1$ is determined by its values at the nodes of $S^E_{\mathcal{N}}$. Hence $K := \dim V_h=\sum_{E \in \mathcal{T}_h}\bigl|S^E_{\mathcal{N}}\bigr|$ 
and we have one degree of freedom for each pair $(E,\bx)$ with $E\in\mathcal{T}_h$ and $\bx\in S^E_{\mathcal{N}}$. Accordingly, we denote the canonical basis of $V_h$ by $\{\phi_{E,\bx}\}$, where $\phi_{E,\bx}$ vanishes outside $E$ and, on $E$, is the function of $V^E_1$ with nodal values $\phi_{E,\bx}(\by)=\delta_{\bx\by}$ for $\by\in S^E_{\mathcal{N}}$. 

For a function $u$ such that its restriction to every element $E$ is continuous, we can define the natural interpolant onto $V_h$ given by
\begin{equation*}\label{Interpolant}
I^h u = \sum_{E\in\mathcal{T}_h}\ \sum_{\bx\in S^E_{\mathcal{N}}} u|_E(\bx)\,\phi_{E,\bx},
\end{equation*}
so that $(I^hu)|_E$ is determined by the trace of $u$ on $\partial E$ taken from within $E$; in particular $I^hu=u$ for every $u\in V_h$. Let 
\begin{align*}
    \hat{b}_h(u, v) =& -  \sum_{e \in \mathcal{E}_h}\left(\langle \ \{\rho\nabla \Pi^{\nabla}_{h, 1}  u \cdot \bn_e\}_w, [v]\rangle_{e} + \langle  \{\rho \nabla \Pi^{\nabla}_{h, 1} v \cdot \bn_e\}_w, [u]\rangle_{e} \right)\\
    &\quad+\sum_{e \in \mathcal{E}_h}\frac{\gamma_h\rho_e }{|e|}\langle [u],[v] \rangle_e, 
\end{align*}
where $\Pi^{\nabla}_{h, 1}$ is the operator which, on each element $E$, coincides with $\Pi^{\nabla}_{E, 1}$; i.e., for $u\in H^s({\cal T}_h)$, $\Pi^{\nabla}_{h, 1}u$ is a discontinuous piecewise linear polynomial. Similarly as in \cite{brezzi} (but with weighted averages instead of aritmetic averages), consider the discrete bilinear form $B_h(\cdot, \cdot): H^s({\cal T}_h) \times H^s({\cal T}_h) \rightarrow \mathbb{R}$ defined by
\begin{align} \label{eq:DVEM_form}
    B_h(u, v) := \hat{a}_h(u, v) + \hat{b}_h(u, v).
\end{align}

Finally, the discrete discontinuous virtual formulation for problem \cref{problemToSolve} is: find $u_h \in V_h$ such that 
\begin{equation}\label{problemDiscret}
    B_h(u_h, v_h) = F_h(v_h) \quad \forall \: v_h \: \in \: V_h.
\end{equation}

\subsection{Continuity and coercivity properties} This section shows that the bilinear form \cref{eq:DVEM_form} is coercive and continuous; hence, the discrete problem \cref{problemDiscret} is well-posed. Related results for DG discretizations of virtual element spaces   can be found in \cite[Theorem 6.2]{brezzi}. 
For $s > 3/2$, define the norm $\|\cdot \|_{1,h}:H^{s}(\mathcal{T}_h) \rightarrow \mathbb{R}$ by
\begin{equation} \label{eq:defNorm}
    \displaystyle \|v\|^2_{1,h} = \sum_{E \in \mathcal{T}_h}(\rho \nabla v,\nabla v)_{E}   + \sum_{e \in \mathcal{E}_h}\frac{\rho_e}{|e|}\langle[v],[v]\rangle_e.
\end{equation}

We use the following trace inequality:
\begin{equation}\label{trace_inequality}
     \|\nabla v\cdot \bn\|_{L^2(e)} \leq C h_E^{-1/2} \|\nabla v\|_{L^2(E)},\quad\forall \: v\in \mathcal{P}_1(E),\: \forall \:e\in\partial E,
\end{equation}
with $C=(\pi\gamma^2)^{-1}$, where $\gamma$ is the constant of Assumption \ref{ass:A1}.

\begin{lemma}\label{lem:auxCoercCont}
There exists $C>0$, depending on only on $\gamma$ and $n_0$ (see Assumptions \ref{ass:A1} and \ref{ass:A3}), such that for all $u,v\in H^s(\mathcal{T}_h)$, $s>3/2$,
\begin{equation*}
    \sum_{e\in\mathcal{E}_h} \langle \{\rho\nabla\Pi_{h,1}^\nabla u\cdot\bn_e\}_w ,[v] \rangle_e \leq C\left(\sum_{e\in\mathcal{E}_h} \frac{\rho_e}{|e|}\|[v]\|_{L^2(e)} ^2\right)^{1/2} \left(\sum_{E\in\mathcal{T}_h} \rho_E \|\nabla v \|_{L^2(E)}^2\right)^{1/2}.
\end{equation*}
In particular, the left-hand side is bounded by $C\|u\|_{1,h}\|v\|_{1,h}.$
\begin{proof}
    For $e\in\mathcal{E}^I$, by \cref{eq:averageWe} and the triangle inequality,
    \begin{align*}
    \| \{\rho\nabla \Pi^{\nabla}_{h, 1}  v \cdot \bn_e\}_w\|_{L^2(e)} &\leq \frac{\rho_e}{2} \left( \|\nabla \Pi^{\nabla}_{E^+}  v \cdot \bn_e \|_{L^2(e)} + \| \nabla \Pi^{\nabla}_{E^-}  v \cdot \bn_e  \|_{L^2(e)} \right).
    \end{align*}    
    Using \cref{trace_inequality}, $\rho_e\leq \sqrt{2}\rho_e^{1/2}(\rho^\pm)^{1/2}$, the stability of $\Pi_E^\nabla$ and Cauchy-Scharwz, we obtain
    \begin{align*}
    \| \{\rho\nabla \Pi^{\nabla}_{h, 1}  v \cdot \bn_e\}_w\|_{L^2(e)} 
    &\leq C\left(\frac{\rho_e}{|e|}\right)^{1/2} \left(\rho^+ \|\nabla v \|_{L^2(E^+)}^2 + \rho^-\| \nabla v\|_{L^2(E^-)}^2 \right)^{1/2}.
    \end{align*}
    A similar bound holds for $e\in \mathcal{E}^B$ with the only element $E^+$. Therefore, using again Cauchy-Schwarz, 
    \begin{align*}
        \sum_{e\in\mathcal{E}_h} &\langle \{\rho\nabla\Pi_{h,1}^\nabla u\cdot\bn_e\}_w, [v] \rangle_e \leq \sum_{e\in\mathcal{E}_h} \| \{\rho\nabla \Pi^{\nabla}_{h, 1} v\cdot \bn_e\}_w\|_{L^2(e)}\|[v]\|_{L^2(e)} \\
        &\leq C\sum_{e\in\mathcal{E}_h}\left(\frac{\rho_e}{|e|}\right)^{1/2} \|[v]\|_{L^2(e)}  \left(\rho^+ \|\nabla v \|_{L^2(E^+)}^2 + \rho^-\| \nabla v\|_{L^2(E^-)}^2 \right)^{1/2} \\
        &\leq C\left(\sum_{e\in\mathcal{E}_h} \frac{\rho_e}{|e|}\|[v]\|_{L^2(e)} ^2\right)^{1/2} \left(n_0\sum_{E\in\mathcal{T}_h} \rho_E \|\nabla v \|_{L^2(E)}^2\right)^{1/2},
    \end{align*}
    which proves the lemma. The second bound follows from \cref{eq:defNorm}.
\end{proof}
\end{lemma}
\begin{lemma}\label{Coercivity}
There exist positive constants $\gamma$ and $\alpha$ depending only on the shape of the elements in $\mathcal{T}_h$, such that if $\gamma_h \geq \gamma$, then
\begin{equation*}\label{continuity}
    B_h(v,v) \geq \alpha\|v\|^2_{1,h} \quad \forall \: v \in \: V_h.
\end{equation*}
\begin{proof}
We need to deduce a lower bound for
\begin{align*}
    B_h(v, v) = \hat{a}_h(v, v) -  2\sum_{e \in \mathcal{E}_h}\langle  \{\rho\nabla \Pi^{\nabla}_{h, 1}  v \cdot \bn_e\}_w, [v]\rangle_{e} +\sum_{e \in \mathcal{E}_h} \frac{\gamma_h\rho_e}{|e|}\langle [v],[v] \rangle_e.
\end{align*}
The first term is easily bounded using \cref{eq:stabIneq}, where we deduce that 
\begin{align*}
    \hat{a}_h(v, v)\geq \hat{C} \sum_{E\in\mathcal{T}_h} (\rho \nabla v,\nabla v)_E,
\end{align*}
where $\hat{C}:= \min\{1,c_1\}$, with $c_1>0$ defined in \cref{eq:stabIneq}. For the second term, we use \cref{lem:auxCoercCont} and the inequality $2ab\leq a^2/\delta + \delta b^2$ for $\delta>0$ to obtain
\begin{align*}
    2\sum_{e \in \mathcal{E}_h}\langle  \{\rho\nabla \Pi^{\nabla}_{h, 1}&  v \cdot \bn_e\}_w, [v]\rangle_{e} \leq \frac{C^2n_0}{\delta} \sum_{e \in \mathcal{E}_h} \frac{\rho_e}{|e|} \|[v]\|_{L^2(e)}^2 + \delta\sum_{E \in \mathcal{T}_h} \rho_E \|\nabla v \|_{L^2(E)}^2.
\end{align*}
Adding all terms, we get
\begin{align*}
     B_h(v, v) \geq \left(\hat{C}-\delta\right)\sum_{E \in \mathcal{T}_h} \rho_E \| \nabla v \|_{L^2(E)}^2 + \sum_{e \in \mathcal{E}_h} \left(\gamma_h- \frac{C^2 n_0}{\delta}\right)\frac{\rho_e}{|e|} \|[v]\|_{L^2(e)}^2.
\end{align*}
Taking $\delta = {\hat{C}}/2$, $\gamma_h \geq \dfrac{2C^2n_0}{\hat{C}} + \dfrac{\hat{C}}{2}$, we conclude that $B_h(v,v)\geq \dfrac{\hat{C}}{2} \|v\|_{1, h}$.
\end{proof}
\end{lemma}

\begin{lemma}\label{Continuity}
Given $u, v \in V_h$, there exists a positive constant $C$, depending only on $\gamma_h$, $n_0$ and the shape regularity of the elements, such that 
\begin{equation*}\label{ContInequality}
    |B_h(u, v)| \leq C\|u\|_{1,h}\|v\|_{1,h}.
\end{equation*}
\begin{proof}
We bound each term of $B_h=\hat{a}_h+\hat{b}_h$ separately. Fix $E\in\mathcal{T}_h$. From standard arguments of the VEM, $\hat{a}^E_h(u,v)\leq \max\{1,c_2\} a^E(u,u)^{1/2}a^E(v,v)^{1/2}$, where $c_2$ is defined in \cref{eq:stabIneq}. Adding over all the elements, it holds that 
\begin{equation}\label{eq_lem1}
|\hat{a}_h(u,v)|\leq \max\{1,c_2\} \|u\|_{1,h}\|v\|_{1,h}.
\end{equation}

Applying \cref{lem:auxCoercCont}, we bound directly
\begin{equation} \label{eq_lem2}
    \sum_{e \in \mathcal{E}_h}\langle \{\rho\nabla \Pi^{\nabla}_{h, 1}  u \cdot \bn_e\}, [v]\rangle_{e} + \langle  \{\rho \nabla \Pi^{\nabla}_{h, 1} v \cdot \bn_e\}, [u]\rangle_{e} \leq 2Cn_0^{1/2} \|u\|_{1,h} \|v\|_{1,h}.
\end{equation}

Finally, the penalty term is bounded by the Cauchy-Schwarz inequality 
\begin{align} \label{eq_lem3}
    \sum_{e \in \mathcal{E}_h}\frac{\gamma_h\rho_e }{|e|}\langle [u],[v] \rangle_e &\leq \gamma_h \left(\sum_{e \in \mathcal{E}_h}\frac{\rho_e }{|e|}\|[u]\|_{L^2(e)}^2\right)^{1/2}\left(\sum_{e \in \mathcal{E}_h}\frac{\rho_e }{|e|}\|[v]\|_{L^2(e)}^2\right)^{1/2}\nonumber\\
    &\leq \gamma_h\|u\|_{1,h}\|v\|_{1,h}.
\end{align}
\end{proof}
The lemma follows from \cref{eq_lem1}, \cref{eq_lem2} and \cref{eq_lem3}.
\end{lemma}

\section{Two level overlapping Schwarz methods} \label{section:2DOS}
In this section, we describe the two-level additive Schwarz methods; for a complete study, see \cite{toselli2006domain}.
We aim to construct a  preconditioner for the ill-conditioned linear system $B \mathbf{\lambda} = \mathbf{b}$ arising from \cref{problemDiscret}.
 
\subsection{Coarse spaces and irregular subdomains}
Consider the partition $\mathcal{T}_H$ formed by the subdomains $\{\Omega_i\}^N_{i=1}$. Then, the discontinuous global virtual space on $\mathcal{T}_H$ is defined as $V_0 := \{v \in H^{s}(\mathcal{T}_H) : v|_{{\Omega}_i} \in V_1^{\Omega_i},\quad 1\leq i \leq N\}$. We have that the dimension of $V_0$ is $\text{dim}V_0=\sum^N_{i=1}|S^{\Omega_i}_{\mathcal{N}}|$, since we have exactly one degree of freedom per every node of each element $\Omega_i$.

A reduced coarse space must be constructed within $V_0$. We consider functions that are continuous across interfaces and have a single degree of freedom associated with each subdomain vertex. The motivation for using a reduced coarse space is to keep the cost of the coarse solver small, mitigating the bottleneck associated with the coarse solver at each step of the preconditioner.

\subsection{Continuous reduced coarse space}For each subdomain vertex $\bx_0 \in S_{\mathcal{V}}$ we define a coarse continuous function $\psi^H_{\bx_0} \in V_0$ by choosing appropriately its degrees of freedom. We use the same construction as presented in \cite[Section 4.1]{CalvoJ}. Using  this construction, we obtain a continuous function $\psi^H_{\bx_0}$ on $\Omega$ such that $0 \leq \psi^H_{\bx_0} \leq 1$ and $\sum_{\bx_0}\psi^H_{\bx_0} = 1$.  We then define the \textit{continuous reduced coarse space} as the span of $\{\psi^H_{\bx_0}\}$, i.e., 
\begin{equation*}
    \displaystyle V^{R,C}_0 := \biggl\{ v \in V_0: v = \sum_{\bx_0 \in S_{\mathcal{V}}}\alpha_{\bx_0}\psi^H_{\bx_0}  \biggr\} \subset V_0,
\end{equation*}
where $\alpha_{\bx_0} \in \mathbb{R} \: \forall \: \bx_0 \in S_{\mathcal{V}}$. We note that the there is only one degree of freedom per subdomain vertex and it is independent of the subdomain to which it belongs. Also, we note that functions of $V^{R,C}_0$ are piecewise linear on the subdomain edges and harmonic in the interior of each subdomain. 

\begin{remark}\label{rem:discreteCoarse}
A discontinuous reduced coarse space $V^{R,D}_0$ can be defined in the same way by not identifying the degrees of freedom associated with a subdomain vertex across the subdomains sharing it. We have $V_0^{R,C}\subset V_0^{R,D}\subset V_0$. The analysis of this paper does not cover this variant. We show a single experiment with $V_0^{R,D}$, reported in \textcolor{red}{\cref{tab:discV0}}, and observed no gain: the coarse space is three to four times larger, while the condition numbers and the iteration counts are essentially the same as those obtained with $V_0^{R,C}$. There is a further advantage in the continuous choice. If the coarse solver is taken to be a discontinuous Galerkin problem posed on $\mathcal{T}_H$, with its own bilinear form and penalty parameter, then the fine and the coarse forms have to satisfy a compatibility relation; see \cite{Feng-Karakashian,10.1007/s10915-017-0453-4}. This is why we restrict our attention to the continuous space $V_0^{R,C}$ in what follows.
\end{remark}

\subsection{Extension operators}
We use the operator $R^T_0: V^{R, C}_0 \rightarrow V_h$ constructed in \cite{CalvoJ} to approximate functions in the coarse space by elements in $V_h$. Given an integer $k \geq 2$, let $V_k^i:=V_k^{\Omega_i}$ be the local virtual element subspace of order $k$ and let 
\begin{equation*}\label{eqn:defVk}
V^{k}_H:= \{v \in H_0^s(\mathcal{T}_H): v|_{\Omega_i} \in V^i_k \quad \forall \: 1 \leq i \leq N\}
\end{equation*}
be the global virtual space on $\mathcal{T}_H$ of degree $k$.

Clearly, $V_0^{R,C} \subseteq V_0 \subseteq V^k_H$. Thus, given a coarse function $\psi^H_{\bx_0}$, we can identify it as an element $\tilde{\psi}^H_{\bx_0}\in V_H^k$; see Section 4.2 from \cite{CalvoJ} for details. We then define the action of $R_0^T$ on each basis function of $V_0^{R,C}$. We note that $R_0^T$ depends on $k \geq 2$, but for simplicity we omit this dependence.  We set the degrees of freedom of $R_0^T \psi^H_{\bx_0} \in V_h$ as follows:
\begin{enumerate}[(a)]
    \item If $E \subset \Omega_i$ and $\bx \in S_{\cal N}^E$ is a node in the interior of $\Omega_i$, we set $(R_0^T \psi_{\bx^H_0})|_E(\bx):=(\Pi^{\nabla}_{\Omega_i , k}\tilde{\psi}_{\bx^H_0})(\bx)$.
    \item For all other nodes in $\mathcal{T}_h$ (that belong to the interface), we set $(R_0^T \psi_{\bx^H_0})(\bx):=\psi_{\bx^H_0}(\bx)$.
\end{enumerate}
Therefore, given $u_0 \in V_0^{R,C}$, we can write the restriction of $R_0^T$ to a subdomain $\Omega_i$ as \begin{equation*}\label{R0T_Omegai}
    \displaystyle (R_0^Tu_0)|_{\Omega_i} = I^h(\Pi^{\nabla}_{\Omega_i, k} \tilde{u}_0) + \sum_{\bx \in S^{\Omega_i}_{\mathcal{N}}}(u_0 - \Pi^{\nabla}_{\Omega_i, k} \tilde{u}_0)(\bx)\phi^h_{\bx}.
\end{equation*}
Notice that $R_0^Tu_0$ is continuous and $R_0^Tu_0 = u_0$ on $ \partial \Omega_i$. Finally, we consider the bilinear form defined in $V_0^{R,C} \times V_0^{R,C}$ by $\tilde{B}_{h,0}(u_0, v_0):= B_h(R_0^Tu_0, R_0^Tv_0)$.

\subsection{Local spaces}
For every $1\leq i \leq N$, let $\mathcal{T}_{h,i} : = \{E \in \mathcal{T}_h : E \subset \Omega'_i\}$. 
We consider the local discontinuous virtual spaces $V_i$, defined by 
\begin{equation*}
    V_i := \{v \in H^s_0(\Omega'_i): v|_E \in \mathcal{B}_1(\partial E), \: \Delta v|_E = 0  \quad \forall \: E \in \mathcal{T}_{h,i}\}.
\end{equation*}
Thus, per every element $E \in \mathcal{T}_{h,i}$ we have one degree of freedom for every node of this element in the interior of $\Omega'_i$. We also consider the natural operators $R^T_i : V_i \rightarrow V_h$ given the zero extension from the subdomain $\Omega'_i$ to $\Omega$, $1 \leq i \leq N$. We use exact solvers for the local spaces, i.e., we define the bilinear forms $\tilde{B}_{h,i} : V_i \times V_i \rightarrow \mathbb{R}$ given by $\tilde{B}_{h,i}(u_i, v_i) := B_h(R_i^Tu_i, R_i^Tv_i)$.

\subsection{Algorithm}

In this subsection we define an additive preconditioner for the linear system $B \mathbf{\lambda} = \mathbf{b}$ obtained from $\cref{problemDiscret}$. Let $R_i^T$ denote the matrix representation of the operators $R_i^T$. Let $\tilde{B}_i =  R_iBR_i^T$, $0 \leq i \leq N$, be the stiffness matrices, and then consider the  Schwarz projections $P_i =  R_i^T\tilde{B}^{-1}_iR_iB$, $0 \leq i \leq N$. The additive preconditioned operator is defined by 
\begin{equation*}
    P_{ad}:= \displaystyle\sum_{i=0}^NP_i = B_{ad}^{-1}B, \quad {\text{with}} \quad B^{-1}_{ad} = \sum_{i=0}^{N}R_{i}^T\tilde{B}^{-1}_{i}R_i.
\end{equation*}
Hence, for the preconditioned system $B^{-1}_{ad} B\mathbf{\lambda} = B_{ad}^{-1}\bm{b}$, the main theorem of this paper is presented below, with the proof provided in \Cref{tech_tools}.

\begin{theorem}\label{main_thm}
    The exists a constant $C$, independent of $N$, $H$, $h$, $\delta$ and $\rho$, such that the condition number of the preconditioned system satisfies 
    \begin{equation*}
        \kappa(B^{-1}_{ad}B) \leq C\left(1+\log\dfrac{H}{h}\right)\left(1+\dfrac{H}{\delta}\right),
    \end{equation*}
    where $H/h$ and $H/\delta$ denote their maximum value over all the subdomains.
\end{theorem}

\section{Technical tools} \label{tech_tools}
This section gathers some tools needed in the proof of \cref{main_thm}. We start with a generalization of \cite[Lemma 2.2]{10.2307/4101206}:


\begin{lemma}\label{N_term2}
    Let $\alpha = [\alpha_1,\ldots,\alpha_m]^T\in\mathbb{R}^m$, $\rho = [\rho_1,\ldots,\rho_m]^T\in\mathbb{R}_+^m$. Denote by $\rho_{ij}$ the harmonic mean \cref{eq:harmonicMean} of $\rho_i$ and $\rho_j$. Define the weighted average of $\alpha$ with weights $\rho$ by
    $$ \beta:=\dfrac{\sum_{j=1}^m \rho_j\alpha_j}{\sum_{j=1}^m \rho_j}.$$
    Then, there exists a constant $C>0$, depending on $m$ and possibly $\rho$, such that
    \begin{equation}\label{eq:ineqAverage}
    \sum_{j=1}^m \rho_j |\alpha_j-\beta|^2 \leq C\sum_{j=1}^{m} \rho_{j,j+1} |\alpha_{j+1}-\alpha_j|^2,
    \end{equation}
    where the indices are understood cyclically, that is, $j=m+1$ corresponds to $j=1$. 
    \begin{proof}
        Set $\gamma=\alpha-\beta\mathbf{1}$. Let $D_\rho$ be the $m\times m$ diagonal matrix $D_\rho = \operatorname{diag}(\rho)$. By the definition of $\beta$ we have $\mathbf{1}^TD_\rho\gamma=\mathbf{1}^TD_\rho\alpha-\beta\,\mathbf{1}^TD_\rho\mathbf{1}=0$. Let $L_\rho$ be the $m\times m$ matrix such that  $$\sum_{j=1}^{m}\rho_{j,j+1}\,|\gamma_{j+1}-\gamma_j|^2=\gamma^{T}L_\rho\,\gamma.$$ 
        The non-zero entries of $L_\rho$ are $(L_\rho)_{j,j}=\rho_{j-1,j}+\rho_{j,j+1}$ and $(L_\rho)_{j,j+1}=(L_\rho)_{j+1,j}=-\rho_{j,j+1}$, with $\rho_{0,1}=\rho_{m,m+1}=0$. Hence, $L_\rho$ is a symmetric positive semi-definite matrix with $L_\rho\mathbf{1}=0$. The right-hand side in \cref{eq:ineqAverage} is
        $$\sum_{j=1}^m\rho_{j,j+1}|\alpha_{j+1}-\alpha_j|^2=\alpha^TL_\rho\alpha=(\gamma+\beta\mathbf{1})^TL_\rho(\gamma+\beta\mathbf{1})=\gamma^TL_\rho\gamma .$$
        Hence \cref{eq:ineqAverage} is equivalent to $\gamma^TD_\rho\gamma\leq C \gamma^TL_\rho\gamma$ for all $\gamma\neq 0$ with $\mathbf{1}^TD_\rho\gamma=0$, and the best constant is the reciprocal of 
        \begin{equation}\label{eq:lambda2}
        \lambda_2(L_\rho,D_\rho):=\min\Bigl\{\tfrac{\gamma^TL_\rho\gamma}{\gamma^TD_\rho\gamma}\ :\ \gamma\neq0,\ \mathbf{1}^TD_\rho\gamma=0\Bigr\}; 
        \end{equation}
        i.e., the smallest non-zero eigenvalue of the generalized eigenvalue problem $L_\rho x=\lambda D_\rho x$. 
    \end{proof}
\end{lemma}

For an interior vertex $v$ with $m$ neighbors ordered in a cycle $E_1,\ldots,E_m$, we consider \cref{{eq:ineqAverage}} with weights $\rho_i:=\rho|_{E_i}$ and nodal values $\alpha_i$ on each element. Hence, \cref{eq:ineqAverage} allow us to bound the left term by jumps across the edges. The optimal constant in \cref{eq:ineqAverage} may of course depend on $\rho$. The point of Assumption \ref{ass:monotonoCoef} is that it provides a lower bound for $\lambda_2(L_\rho,D_\rho)$ that does not involve the contrast of the coefficient:

\begin{lemma}\label{N_term_rho}
Let Assumption \ref{ass:monotonoCoef} hold. Then, the optimal constant of \cref{N_term2} satisfies $C_{\mathrm{opt}}\leq m^2$ and is independent of $\rho_1,\ldots,\rho_m$.
\begin{proof}
   Since $\beta$ minimizes $t\mapsto\sum_j\rho_j|\alpha_j-t|^2$ over $t\in\mathbb{R}$, it holds that
    \begin{equation}\label{eq:minProperty}
   \sum_{j=1}^m\rho_j|\alpha_j-\beta|^2\;\leq\;\sum_{j=1}^m\rho_j|\alpha_j-\alpha_{j_{\max}}|^2 .
    \end{equation}
    Fix $j$ and let $j=j_1,\ldots,j_L=j_{\max}$ be a monotone path of indices. Since the harmonic mean of two positive numbers is bounded below by their minimum, $\rho_{j_l}\leq\rho_{j_{l+1}}$ gives
    \begin{equation}\label{eq:harmonicMonotone}
   \rho_j=\rho_{j_1}\leq\rho_{j_l}=\min\{\rho_{j_l},\rho_{j_{l+1}}\}\leq\rho_{j_l,j_{l+1}},
   \qquad 1\leq l\leq L-1 .
    \end{equation}
    Writing $\alpha_j-\alpha_{j_{\max}}$ as the telescoping sum of the increments along the path, the Cauchy-Schwarz inequality and \cref{eq:harmonicMonotone} yield
    $$\rho_j|\alpha_j-\alpha_{j_{\max}}|^2\leq(L-1)\sum_{l=1}^{L-1}\rho_j\bigl|\alpha_{j_{l+1}}-\alpha_{j_l}\bigr|^2 \leq m\sum_{l=1}^{L-1}\rho_{j_l,j_{l+1}}\bigl|\alpha_{j_{l+1}}-\alpha_{j_l}\bigr|^2,$$
    where $j_{\max}$ is an index for which $\rho_{j_{\max}} = \max_{j} \rho_j$. 
    Summing over $j$ and using that each edge of the patch belongs to at most $m$ of the paths, we obtain $\sum_j\rho_j|\alpha_j-\alpha_{j_{\max}}|^2\leq m^2\sum_{j=1}^m\rho_{j,j+1}|\alpha_{j+1}-\alpha_j|^2$, which together with \cref{eq:minProperty} proves the lemma. 
\end{proof}
\end{lemma}

\begin{remark}\label{rem:qmSharp}
Assumption \ref{ass:monotonoCoef} cannot be dispensed with in \cref{N_term_rho}. Let $\bx$ be an interior node surrounded by four elements, with $(\rho_1,\rho_2,\rho_3,\rho_4)=(M,1,M,1)$ and $M\geq1$; this is the configuration produced by a checkerboard coefficient around an interior vertex of a structured partition that is not quasi-monotone. All four edges carry the same harmonic mean $\rho_{j,j+1}=2M/(M+1)$ and a direct computation shows that the generalized eigenvalues are $\{0, 4/(M+1), 4M/(M+1), 4\}$. We then have that in \cref{N_term2} the optimal constant is $C_{\mathrm{opt}}=\frac{M+1}{4}$, which grows linearly in the contrast $M$, and no bound independent of the contrast can hold.
\end{remark}

\begin{remark}\label{rem:closingEdge}
In \cite{10.2307/4101206}, a bound of the form 
\begin{equation}\label{eq:ineqAverageB}
    \sum_{j=1}^m |\alpha_j-\beta|^2 \leq C \sum_{j=1}^{m-1} |\alpha_{j+1}-\alpha_j|^2
\end{equation}
was proved, where $C$ depends only on $m$. This inequality was used for a discontinuous formulation of our problem for finite element discretizations with $\rho=1$. For a variable coefficient, the closing term must not be discarded as in \cref{eq:ineqAverageB}. Let $m=5$ and $(\rho_1,\ldots,\rho_5)=(M,1,1,1,M)$ with $M\geq1$. This configuration satisfies Assumption \ref{ass:monotonoCoef}. Accordingly \cref{N_term_rho} applies, and the optimal constant of \cref{eq:ineqAverage} is
\begin{equation*}
   C_{\mathrm{opt}}=\frac{2}{5-\sqrt{(9M+1)/(M+1)}}\leq 1.
\end{equation*}
If instead the term $j=m$ is dropped as in \cref{eq:ineqAverageB}, the optimal constant now is 
\begin{equation*}
   C_{\mathrm{opt}}=\frac{(M+1)\bigl(3+\sqrt{9-8/(M+1)}\bigr)}{4}=\frac{3}{2}M+O(1),
\end{equation*}
and the estimate degenerates linearly in the contrast; see \cref{fig:patch}.
\end{remark}

\tikzset{
  mesh/.style  ={draw=black!40, line width=0.6pt},
  outer/.style ={draw=black!40, line width=0.9pt},
  eltlab/.style={font=\footnotesize, black!60},
  gnode/.style ={circle, fill=black, inner sep=1.6pt},
  glink/.style ={line width=1.2pt},
  gdrop/.style ={line width=1.2pt, densely dashed, draw=black!35},
  wlab/.style  ={font=\scriptsize, fill=white, inner sep=1.2pt},
  vlab/.style  ={font=\small},
  nuv/.style   ={circle, draw=black, fill=white, line width=0.7pt, inner sep=1.6pt},
  capt/.style  ={font=\small, align=center},
}
\newcommand{\vempatch}[2]{%
  \def\R{2.45}%
  \foreach \k in {0,...,4} {\pgfmathsetmacro\A{90+72*\k} \coordinate (Q\k) at (\A:\R);}%
  \coordinate (c) at (0,0);%
  \foreach \k in {0,...,4} {\pgfmathtruncatemacro\kk{mod(\k+1,5)}%
     \draw[mesh] (c)--(Q\k); \draw[outer] (Q\k)--(Q\kk);}%
  \foreach \k in {0,...,4} {\pgfmathsetmacro\A{90+72*\k+36} \pgfmathtruncatemacro\j{\k+1}%
     \node[gnode] (D\k) at (\A:{0.50*\R}) {};%
     \node[eltlab] at (\A:{1.16*\R}) {$E_\j$};%
     \ifnum#2=1 \node[vlab] at (\A:{0.82*\R}) {$\alpha_\j$};%
     \else       \node[vlab] at (\A:{0.82*\R}) {$\rho_\j$};\fi}%
  \foreach \k in {0,...,3} {\pgfmathtruncatemacro\kk{\k+1} \draw[glink] (D\k)--(D\kk);}%
  \ifnum#1=1 \draw[glink] (D4)--(D0); \else \draw[gdrop] (D4)--(D0);\fi%
  \ifnum#2=0%
    \foreach \k/\lb in {0/{\rho_{1,2}},1/{\rho_{2,3}},2/{\rho_{3,4}},3/{\rho_{4,5}}}%
      {\pgfmathsetmacro\A{90+72*\k+72} \node[wlab] at (\A:{0.445*\R}) {$\lb$};}%
    \pgfmathsetmacro\A{90+72*4+72} \node[wlab] at (\A:{0.445*\R}) {$\rho_{5,1}$};%
  \fi%
  \node[nuv] at (c) {}; \node[font=\scriptsize] at (-0.29,-0.15) {$\bx$};%
}

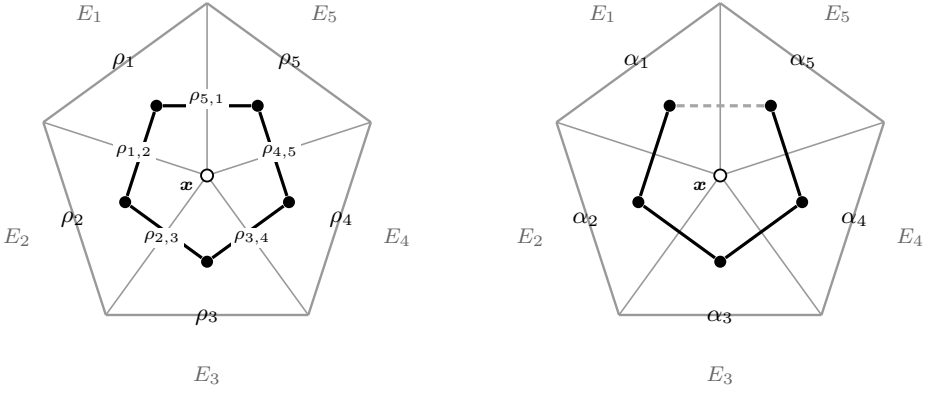
\begin{figure}[t!]
\centering
\begin{tikzpicture}[scale=0.93]
  \begin{scope}
    \vempatch{1}{0}
  \end{scope}
  \begin{scope}[xshift=7.3cm]
    \vempatch{0}{1}
  \end{scope}
\end{tikzpicture}
\caption{The patch $\mathcal{T}_{\bx,h}$ of an interior node $\bx$ with $m=5$ elements (in grey), and the graph (in black) underlying \cref{eq:ineqAverage} and \cref{eq:ineqAverageB}. (left) For the full sum in \cref{N_term2} with harmonic weights, the constant in independent of the contrast (right) If the edge shared by $E_5$ and $E_1$ is discarded in \cref{eq:ineqAverage}, the bound can be dependent of the contrast even if we use harmonic means; see \cref{rem:closingEdge}.}
\label{fig:patch}
\end{figure}

We now prove the existence of a continuous virtual element function that interpolates any function $u\in V_h$. This lemma allow us to define and bound the coarse component of a stable decomposition for $u$. We first prove the following equivalence:

\begin{lemma}
Let $e$ be an edge with endpoints $\bx_1$, $\bx_2$. For $r\in\mathcal{P}_1(e)$ it holds that
\begin{equation}\label{eq:P1equiv}
  \tfrac{|e|}{6}\bigl(r(\bx_1)^2+r(\bx_2)^2\bigr)\;\leq\;\|r\|^2_{L^2(e)}\;\leq\;\tfrac{|e|}{2}\bigl(r(\bx_1)^2+r(\bx_2)^2\bigr).
\end{equation}
    \begin{proof}
    Since $3\|r\|^2_{L^2(e)}=|e|\left(r(\bx_1)^2+r(\bx_1)r(\bx_2)+r(\bx_2)^2\right)$, 
    the lemma holds after a straightforward computation.
    \end{proof}
\end{lemma}

\begin{lemma}\label{lemmachi}Let $u \in V_h$. There exists a continuous piecewise virtual function $\chi \in V_h \cap H_0^1(\Omega)$ and a constant $C > 0$ such that 
\begin{align}
    \label{sumL2}\sum_{E\in\mathcal{T}_h} \rho_E\|u-\chi\|_{L^2(E)}^2 &\leq C \sum_{e\in\mathcal{E}_{h}} \rho_e h_e\|[u]\|_{L^2(e)}^2,\\
    \label{sumh1}
    \sum_{E\in\mathcal{T}_h}\rho_E \|\nabla(u-\chi)\|_{L^2(E)}^2 &\leq C \sum_{e\in\mathcal{E}_{h}} \rho_eh_e^{-1}\|[u]\|_{L^2(e)}^2,
\end{align}
where $C$ depends only on the shape regularity of the elements of ${\cal T}_h$.
\end{lemma}
\begin{proof}
We construct $\chi \in V_h\cap H_0^1(\Omega)$ similarly as in \cite{10.1007/s10915-017-0453-4} for the case of triangular meshes, except that the degrees of freedom of $\chi$ are defined as the $\rho$-weighted averages of the nodal values of the discontinuous function $u\in V_h$. 

For $u\in V_h$ we write $u=\sum_{E\in\mathcal{T}_h}\ \sum_{\bx\in S^E_{\mathcal{N}}}\alpha_E^{(\bx)}\,\phi_{E,\bx}$ with $\alpha_E^{(\bx)}=u|_E(\bx)$, so that $\alpha_E^{(\bx)}$ is the value that $u$ takes at the node $\bx$ as seen from the element $E$. We define
\begin{equation*}
    \chi:=\sum_{\bx\in\mathcal{N}_v}\beta^{(\bx)}\,\phi^h_{\bx}, \qquad \beta^{(\bx)}:=\frac{\sum_{E\in\mathcal{T}_{\bx,h}}\rho_E\,\alpha_E^{(\bx)}}{\sum_{E\in\mathcal{T}_{\bx,h}}\rho_E},
\end{equation*}
where $\phi^h_{\bx}:=\sum_{E\,:\,\bx\in S^E_{\mathcal{N}}}\phi_{E,\bx}$ is the nodal continuous function with $\phi^h_{\bx}(\bx)=1$ that vanishes at all the remaining vertices of the mesh. Clearly $\chi\in V_h\cap H^1_0(\Omega)$. Setting $\beta^{(\bx)}:=0$ for $\bx\in\mathcal{N}_b$, this reads $\chi=\sum_{E\in\mathcal{T}_h}\ \sum_{\bx\in S^E_{\mathcal{N}}}\beta^{(\bx)}\,\phi_{E,\bx}$.

From Assumptions \ref{ass:A1} and \ref{ass:A3}, we have that $\|\phi^h_{\bx}\|_{L^2(E)}^2\leq C_1 h_E^2$. Therefore, grouping the terms by nodes,
\begin{align*}
    \sum_{E\in\mathcal{T}_h}\rho_E \|u-\chi\|_{L^2(E)}^2 \leq  C_1 \max_{E \in \mathcal{T}_h} (d_E)
    \sum_{\bx\in S_{\mathcal{N}}} h_{\bx}^2 \sum_{E\in\mathcal{T}_{\bx,h}}
    \rho_E\bigl|\alpha_E^{(\bx)}-\beta^{(\bx)}\bigr|^2,
\end{align*}
with $h_{\bx} = \max_{E\in\mathcal{T}_{\bx,h}} h_E$.

Fix an interior node $\bx\in\mathcal{N}_v$ and let $E_1,\ldots,E_m$ be the elements of $\mathcal{T}_{\bx,h}$ enumerated cyclically such that $E_i$ and $E_{i+1}$ share an edge $e_i\in\mathcal{E}_h$ having $\bx$ as an endpoint. Writing $\alpha_i:=\alpha_{E_i}^{(\bx)}$, by \cref{N_term2} we have that
\begin{equation*}
   \sum_{E\in\mathcal{T}_{\bx,h}}\rho_E\bigl|\alpha_E^{(\bx)}-\beta^{(\bx)}\bigr|^2
   \leq m^2\sum_{i=1}^{m}\rho_{e_i}\bigl|\alpha_{i+1}-\alpha_i\bigr|^2 ,
\end{equation*}
where $\rho_{e_i}$ is the harmonic mean of the coefficients of $E_i$ and $E_{i+1}$. The difference $\alpha_{i+1}-\alpha_i$ is the jump of $u$ at $\bx$ across $e_i$, that is $[u]|_{e_i}(\bx)$. For a boundary node $\bx\in\mathcal{N}_b$ we have $\beta^{(\bx)}=0$ and the same bound holds. Summing over the nodes and regrouping by edges,
\begin{equation*}
   \sum_{E\in\mathcal{T}_h}\rho_E\|u-\chi\|^2_{L^2(E)}
   \leq C\sum_{e\in\mathcal{E}_h}\rho_e\sum_{\bx\in e}h_{\bx}^2\bigl|[u]|_e(\bx)\bigr|^2 .
\end{equation*}
Finally, the two endpoints of $e$ satisfy $h_{\bx}\leq Ch_e$ by Assumption \ref{ass:A2}; since $[u]|_e\in\mathcal{P}_1(e)$, the lower bound in
\cref{eq:P1equiv} together with $|e|\approx h_e$ from Assumption \ref{ass:A3} gives
\begin{equation*}
   \sum_{\bx\in e} h_{\bx}^2\bigl|[u]|_e(\bx)\bigr|^2
   \leq C h_e^2\sum_{\bx\in e}\bigl|[u]|_e(\bx)\bigr|^2
   \leq C\,\frac{h_e^2}{|e|}\,\|[u]\|_{L^2(e)}^2
   \leq C\,h_e\|[u]\|_{L^2(e)}^2,
\end{equation*}
and we obtain \cref{sumL2}. Since $|\phi^h_{\bx}|_{H^1(E)}^2\leq C_2$, \cref{sumh1} is deduced similarly. 
\end{proof}

\begin{lemma}\label{lem:jumpInterp}
Let $\{\theta_i\}_{i=1}^N\subset C^0(\overline{\Omega})$ satisfy $0\leq\theta_i\leq1$ and
$\sum_{i=1}^N\theta_i=1$. For every $v\in V_h$ it holds that
\begin{equation}\label{eq:jumpInterpSum}
   \sum_{i=1}^N\|[I^h(\theta_i v)]\|^2_{L^2(e)}\leq 3\|[v]\|^2_{L^2(e)}
   \qquad\forall e\in\mathcal{E}_h.
\end{equation}
\begin{proof}
Fix and edge $e$ and let $\bx_1,\bx_2$ be its endpoints. Set $q:=[v]|_e$ and $p_i:=[I^h(\theta_i v)]|_e$. Since $v,I^h(\theta_i v)\in V_h$, then $p_i,q\in\mathcal{P}_1(e)$. As $\theta_i$ is single-valued at each node, the degrees of freedom of $I^h(\theta_i v)$ associated with $\bx_l$ and with $E^{\pm}$ are $\theta_i(\bx_l)v^{\pm}(\bx_l)$, so that
$$p_i(\bx_l)=\theta_i(\bx_l)v^+(\bx_l)-\theta_i(\bx_l)v^-(\bx_l)=\theta_i(\bx_l)\,q(\bx_l),\qquad l\in\{1,2\},$$
and likewise for $e\in\mathcal{E}^B$. By the upper bound in \cref{eq:P1equiv} and
$\sum_{i}\theta_i(\bx)^2\leq\sum_{i}\theta_i(\bx)=1$, then
$$\sum_{i=1}^N\|p_i\|^2_{L^2(e)}\leq\tfrac{|e|}{2}\sum_{l=1}^2\Bigl(\sum_{i=1}^N\theta_i(\bx_l)^2\Bigr)q(\bx_l)^2\leq\tfrac{|e|}{2}\sum_{l=1}^2 q(\bx_l)^2,$$
and the lower bound in \cref{eq:P1equiv} gives
$\dfrac{|e|}{2}\sum_{l}q(\bx_l)^2\leq3\|q\|^2_{L^2(e)}$.
\end{proof}
\end{lemma}

Let $\omega_i:= 1 +\log(H_i/h_i)$ and $\omega := \displaystyle\max_{i}\omega_i$. We now recall from \cite{CalvoJ} the two estimates for the coarse operators that we use. Both are stated there for the conforming virtual element space $V_h\cap H^1_0(\Omega)$. We will apply these results to the function $\chi$ of \cref{lemmachi}.

\begin{lemma}[{\cite[Lemma 5.2]{CalvoJ}}]\label{lem:discreteSobolev}
Let $v\in V_h\cap H^1_0(\Omega)$ have zero mean value on $\Omega_i$. Then
\begin{equation}\label{eq:discreteSobolev}
\|v\|^2_{L^\infty(\Omega_i)}\;\leq\;C\,\omega_i\,|v|^2_{H^1(\Omega_i)},
\end{equation}
with $C$ independent of $H_i$ and $h_i$.
\end{lemma}

\begin{lemma}[{\cite[Lemma 5.7]{CalvoJ}}]\label{u-u0}
    Let $\chi\in V_h\cap H^1_0(\Omega)$ and set 
    $$u_0:=\sum_{\bx_0\in S_{\mathcal{V}}}\chi(\bx_0)\psi^H_{\bx_0}\in V_0^{R,C}.$$
    If $k\geq2$, then for $1\leq i\leq N$
    \begin{align}
    \|\chi-R_0^Tu_0\|^2_{L^2(\Omega_i)} &\leq C\,\omega_i H_i^2\,|\chi|^2_{H^1(\Omega_i)}, \quad     |R_0^Tu_0|^2_{H^1(\Omega_i)} \leq C\,\omega_i\,|\chi|^2_{H^1(\Omega_i)},\label{h1u0}
    \end{align}
    where $C$ is independent of $H_i,h_i,\delta_i$ and $N$.
\end{lemma}

\begin{lemma}[Stable decomposition] Given $u \in V_h$, there exist functions $u_i\in V_i$ for $i\in\{0,1,\ldots N\}$ and a constant $C_0$ independent of $H,h,N,\rho$ such that $u= \sum_{i=0}^N R_i^T u_i$ and
\begin{equation}\label{eq:C0}
   \sum_{i=0}^N\tilde{B}_{h,i}(u_i, u_i) \leq C^2_0\,B_h(u,u), \qquad C_0^2\leq C\,\omega\Bigl(1+\frac{H}{\delta}\Bigr).
\end{equation}
\end{lemma}
\begin{proof}
Consider $u_0  \in V_0^{R,C}$ as in \cref{u-u0}. Let $w := u - R_0^Tu_0 \in V_h$ and define $u_i=R_i(I^h(\tilde{\theta}_i w )) \in V_i$, $i=1,...,N$, where \(\{\tilde{\theta}_i\}_{i=1}^N\) is a typical partition of unity for the overlapping subdomains \(\Omega'_i\); see \cite{toselli2006domain}. 

For the  coarse component, since it is a continuous function, all jumps vanish and it holds that
\begin{align*}
    \tilde{B}_{h,0}(u_0, u_0) &= B_h(R_0^T u_0, R_0^T u_0)= \sum_{i=1}^N \hat{a}_h^{\Omega_i}(R_0^T u_0, R_0^T u_0)\leq C \sum_{i=1}^N \rho_{\Omega_i} |R_0^T u_0|^2_{H^1(\Omega_i)}\\
    &\leq C\sum_{i=1}^N \rho_{\Omega_i}\omega_i |\chi|^2_{H^1(\Omega_i)}\leq C \omega B_h(u,u),
\end{align*}
where we have used continuity of the VEM bilinear form, \cref{u-u0} and \cref{lemmachi}.

From \cref{Continuity}, we can bound the contribution of the local components by
\begin{equation}\label{local_bi}
    B_h(R_i^T u_i, R_i^T u_i) \leq C \sum_{E \subset \Omega_i'} \rho_E\left|I^h(\tilde{\theta}_i w)\right|_{H^1(E)}^2 + C \sum_{e \in \mathcal{E}_h} \frac{\rho_e}{|e|} \left\|[I^h(\tilde{\theta}_i w)]\right\|^2_{L^2(e)}.
\end{equation}

To bound the first summation, note that for each \(E \subset \Omega_i'\),
\begin{equation}\label{first_term}
    \rho_E\left| \, I^h(\tilde{\theta}_i w)\right|_{H^1(E)}^2 
    \leq C \left(\rho_E\left\|I^h(\tilde{\theta}_i) \nabla I^h(w)\right\|_{L^2(E)}^2 
    + \rho_E\left\|I^h(w) \nabla I^h(\tilde{\theta}_i)\right\|_{L^2(E)}^2 \right).
\end{equation}

Using that $I^hw=w$ and $0\leq I^h(\tilde{\theta}_i)\leq1$, we obtain 
\begin{align*}
\sum_{E \subset \Omega_i'} \rho_E&\left\|I^h(\tilde{\theta}_i)\nabla I^h(w)\right\|_{L^2(E)}^2  \leq \sum_{E \in {\cal T}_h} \rho_E\left| w\right|_{H^1(E)}^2 \\
&\leq 2\|u\|_{1,h}^2+2\sum_{i=1}^N\rho_{\Omega_i}\left|R_0^Tu_0\right|_{H^1(\Omega_i)}^2
\leq C\omega\|u\|_{1,h}^2.    
\end{align*}

Here the second inequality follows from $w=u-R_0^Tu_0$ and the triangle inequality, while the third one follows from \cref{h1u0} together with $\sum_{i}\rho_{\Omega_i}|\chi|^2_{H^1(\Omega_i)}\leq C\|u\|^2_{1,h}$, which is a consequence of \cref{sumh1}.  

For the second term in \cref{first_term}, let $\Xi_i:=\{j:\overline{\Omega_j}\cap\overline{\Omega_i}\neq\emptyset\}$, so that $i\in\Xi_i$. Then, $\tilde{\theta}_i$ equals one on $\Omega_i\setminus\bigcup_{j\neq i}\Omega'_j$ and vanishes outside $\Omega'_i$. Hence, $\nabla I^h(\tilde{\theta}_i)$ vanishes outside a neighbourhood of $\partial\Omega_i$ of width $\max_{j\in\Xi_i}\delta_j$ and $\|\nabla I^h(\tilde\theta_i)\|_{L^\infty}\leq C/\delta_i$. Since the number of sets $\Omega_j'$ that intersect $\Omega_i$ is uniformly bounded, it suffices to estimate the contribution of a single $j$. We split 
\begin{equation}\label{eq:splitW}
   w = \underbrace{(u-\chi)}_{=:w_1} +\underbrace{(\chi-u_0)}_{=:w_2} +\underbrace{(u_0-R_0^Tu_0)}_{=:w_3},
\end{equation}
and bound the three contributions separately.

For $w_1$, the elements meeting $\Omega'_j\cap\Omega_i$ lie in a layer of width $C\delta_i$, so by the local form of \cref{sumL2} and $h_e\leq h_i\leq\delta_i$,
\begin{align}
  \frac{\rho_{\Omega_i}}{\delta_i^2}\|w_1\|^2_{L^2(\Omega'_j\cap\Omega_i)} &\leq\frac{C}{\delta_i^2}\sum_{e}\rho_eh_e\|[u]\|^2_{L^2(e)} = C\sum_{e}\frac{h_e^2}{\delta_i^2}\frac{\rho_e\|[u]\|^2_{L^2(e)}}{h_e} \nonumber \\
  &\leq C\sum_{e}\frac{\rho_e}{|e|}\|[u]\|^2_{L^2(e)},\label{eq:w1}
\end{align}
the sums being over the edges of the elements meeting $\Omega'_j\cap\Omega_i$.

For $w_2$, let $\hat{\chi}$ be the mean value of $\chi$ over $\Omega_i$. Since $u_0$ is harmonic in $\Omega_i$ and its degrees of freedom are uniformly bounded by the nodal values of $\chi$, we have $\|w_2\|_{L^\infty(\Omega_i)}\leq C\|\chi-\hat{\chi}\|_{L^\infty(\Omega_i)}$, and \cref{lem:discreteSobolev} together with $|\Omega'_j\cap\Omega_i|\leq CH_i\delta_i$ gives
\begin{equation}\label{eq:w2}
  \frac{\rho_{\Omega_i}}{\delta_i^2}\|w_2\|^2_{L^2(\Omega'_j\cap\Omega_i)} \leq\frac{C\rho_{\Omega_i}}{\delta_i^2}\|\chi-\hat{\chi}\|^2_{L^\infty(\Omega_i)}\, |\Omega'_j\cap\Omega_i| \leq C\,\frac{H_i}{\delta_i}\,\omega_i\,\rho_{\Omega_i}|\chi|^2_{H^1(\Omega_i)} .
\end{equation}
This is the only term that produces the factor $H/\delta$.

For $w_3$, we use a similar argument as in \cite[proof of Theorem 4.1]{CalvoJ}. By the definition of $R_0^T$ we have $R_0^Tu_0=u_0$ on $\partial\Omega_i$, so that $w_3$ vanishes on $\partial\Omega_i$. Cover $\Omega'_j\cap\Omega_i$ with square patches $\pi_1,\ldots,\pi_L$ of side of order $\delta_i$ and with bounded overlap; since the region is a layer of width $C\delta_i$ along $\partial\Omega_i$, $L$ is of order $H_i/\delta_i$ and every $\pi_k$ meets $\partial\Omega_i$. Friedrichs' inequality on $\pi_k$, applicable because $w_3$ vanishes on $\partial\pi_k\cap\partial\Omega_i$, yields $\|w_3\|^2_{L^2(\pi_k)}\leq C\delta_i^2|w_3|^2_{H^1(\pi_k)}$, whence
\begin{equation}\label{eq:w3}
  \frac{\rho_{\Omega_i}}{\delta_i^2}\|w_3\|^2_{L^2(\Omega'_j\cap\Omega_i)} \leq\frac{C\rho_{\Omega_i}}{\delta_i^2}\sum_{k=1}^L\|w_3\|^2_{L^2(\pi_k)} \leq C\rho_{\Omega_i}\sum_{k=1}^L|w_3|^2_{H^1(\pi_k)} \leq C\rho_{\Omega_i}|w_3|^2_{H^1(\Omega_i)},
\end{equation}
and $|w_3|^2_{H^1(\Omega_i)}\leq 2|\chi-u_0|^2_{H^1(\Omega_i)}+2|\chi-R_0^Tu_0|^2_{H^1(\Omega_i)}\leq C\omega_i|\chi|^2_{H^1(\Omega_i)}$ by \cref{h1u0}.

On the other hand, the second summation in \cref{local_bi} is handled by \cref{lem:jumpInterp}. After multiplying \cref{eq:jumpInterpSum} by $\rho_e/|e|$ and summing over all edges,
\begin{equation}\label{eq:jumpsAssembled}
  \sum_{i=1}^N\sum_{e\in\mathcal{E}_h} \frac{\rho_e}{|e|}\bigl\|[I^h(\tilde{\theta}_i w)]\bigr\|^2_{L^2(e)} \;\leq\; 3\sum_{e\in\mathcal{E}_h}\frac{\rho_e}{|e|}\bigl\|[w]\bigr\|^2_{L^2(e)}=3\sum_{e}\frac{\rho_e}{|e|}\|[u]\|^2_{L^2(e)}\leq3\|u\|^2_{1,h},
\end{equation}
where we used the fact that $[w]=[u]$ since $R_0^Tu_0$ is continuous on $\Omega$ and vanishes on $\partial\Omega$.

Collecting \cref{eq:w1}, \cref{eq:w2}, \cref{eq:w3} and \cref{eq:jumpsAssembled}, summing over $i$ and using that $$\sum_i\rho_{\Omega_i}|\chi|^2_{H^1(\Omega_i)}\leq C\|u\|^2_{1,h}$$ by \cref{sumh1}, together with the bounded overlap of the sets $\Omega_i'$, we obtain \cref{eq:C0} with $C_0^2\leq C\omega(1+H/\delta)$, and the proof is complete.

\begin{proof}[Proof of \cref{main_thm}]
From the abstract additive Schwarz theory \cite[Theorem 2.7]{toselli2006domain}, we have that 
$$C_0^{-2}\,B_h(u,u)\;\leq\;B_h(P_{ad}u,u)\;\leq\; \bigl(N^C+1\bigr)\,B_h(u,u)\qquad\forall\,u\in V_h,$$
where $N^C$ is the minimum number of colors needed to paint the overlapping subdomains $\Omega_i'$ such that no pair of subdomains of the same color intersect. Therefore $\kappa(B_{ad}^{-1}B)\leq C_0^2(N^C+1)$. The conclusion follows from standard arguments; see \cite[Theorem 2.7, Lemma 2.10]{toselli2006domain}.
\end{proof}

\end{proof}

\section{Experimental results}\label{section:experiments}
In this section we report numerical results for the discretization of \Cref{section:DGM+VEM} and for the two-level additive preconditioner of \Cref{section:2DOS}. All the experiments are for the model problem \cref{problemToSolve} on the unit square $\Omega=(0,1)^2$. \Cref{sec:convergence} checks that the discretization converges at the optimal rate and the remaining subsections study the preconditioner. The preconditioned systems are solved with the preconditioned conjugate gradient method, starting from a zero initial guess and stopping when the relative residual is below $10^{-6}$. We use the penalty parameter $\gamma_h=10$ and the projector $\Pi^{\nabla}_{\Omega_i,k}$ with $k\in\{2,4\}$ in the definition of the extension operator $R_0^T$, and harmonic extensions as well.

We consider three families of fine meshes: structured triangular meshes, hexagonal meshes and Voronoi meshes, and two families of subdomain partitions. In the first one, $\Omega$ is divided into $N$ congruent square subdomains; this is the situation covered by the classical theory, and we include it as a reference. In the second one, the subdomains are generated with the graph partitioner METIS, so that they are unions of fine elements with no prescribed shape. In all the tables, $\dim V_0$ denotes the dimension of the coarse space, $I$ the number of PCG iterations and $\kappa$ the condition number of the preconditioned system.

\subsection{Convergence of the discontinuous virtual element method} \label{sec:convergence}

We first check that the discretization of \Cref{section:DGM+VEM} attains optimal convergence rates on three families of meshes. We solve \cref{problemToSolve} with homogeneous Dirichlet boundary conditions for the following values of $\rho$:
\begin{enumerate}[(i)]
    \item[(i)] $\rho\equiv1$,
    \item[(ii)] $\rho = 1+10xy$,
    \item[(iii)] $\rho = 2+\cos(10\pi x)\cos(10\pi y)$, and
    \item[(iv-v)] $\rho = 1$ for $x<1/2$ and $\rho = 10$ for $x>1/2$.
\end{enumerate}
For the cases (i-iv), the exact solution is $u(x,y)=\sin(\pi x)\sin(\pi y)$ and $f$ is computed accordingly. In case (iv) this function solves the interface problem because $\partial_x u$ vanishes on $x=1/2$ and therefore so does the jump of the normal flux. Case (v) is included as a sharper test: there we take $u(x,y)=\varphi(x)\sin(\pi y)$ with
\begin{equation*}
    \varphi(x) = \begin{cases}
        x, & 0\leq x\leq 1/2,\\
        \displaystyle \frac{11}{5}\left(x-\dfrac{1}{2}\right)^2+\dfrac{1}{10}\left(x-\dfrac{1}{2}\right)+\dfrac{1}{2}, & 1/2 < x \leq 1,
    \end{cases}
\end{equation*}
so that $\varphi$ and $\rho\varphi'$ are continuous across $x=1/2$, $\varphi(0)=\varphi(1)=0$, and the normal flux there equals $\sin(\pi y)\neq0$. We take $\gamma_h=10$.

Since the analysis assumes that $\rho$ is constant on each element, for cases (iv) and (v) we cut the elements accordingly. We report the relative errors
\begin{equation*}
    \frac{\|u-\Pi_h^{\nabla}u_h\|_{L^2(\Omega)}}{\|u\|_{L^2(\Omega)}}
    \qquad\text{and}\qquad
    \frac{\bigl(\sum_{E\in\mathcal{T}_h}\|\nabla u-\nabla\Pi^{\nabla}_{E,1}u_h\|^2_{L^2(E)}\bigr)^{1/2}}{|u|_{H^1(\Omega)}}.
\end{equation*}

\Cref{fig:coefficients} shows the last three piecewise constant coefficients and \cref{fig:solutions} displays the corresponding discrete solutions on a coarse mesh of each family. The last column of \cref{fig:solutions} is the cut of case (v) along $y=1/2$ to show that the kink of the exact solution at the interface.

\begin{figure}[H]
    \centering
    \includegraphics[width=\linewidth]{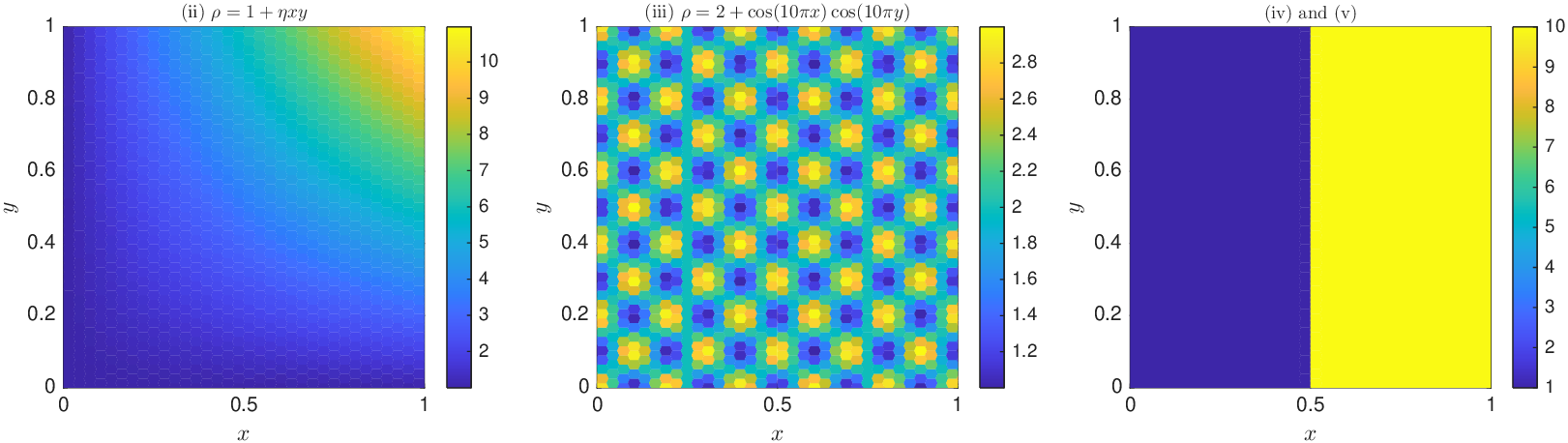}
    \caption{The coefficients (ii), (iii) and (iv)-(v) of \Cref{sec:convergence}. Cases (iv) and (v) share the same coefficient and case (i) is constant.}
    \label{fig:coefficients}
\end{figure}

\begin{figure}[h!]
    \centering
    \includegraphics[width=\linewidth]{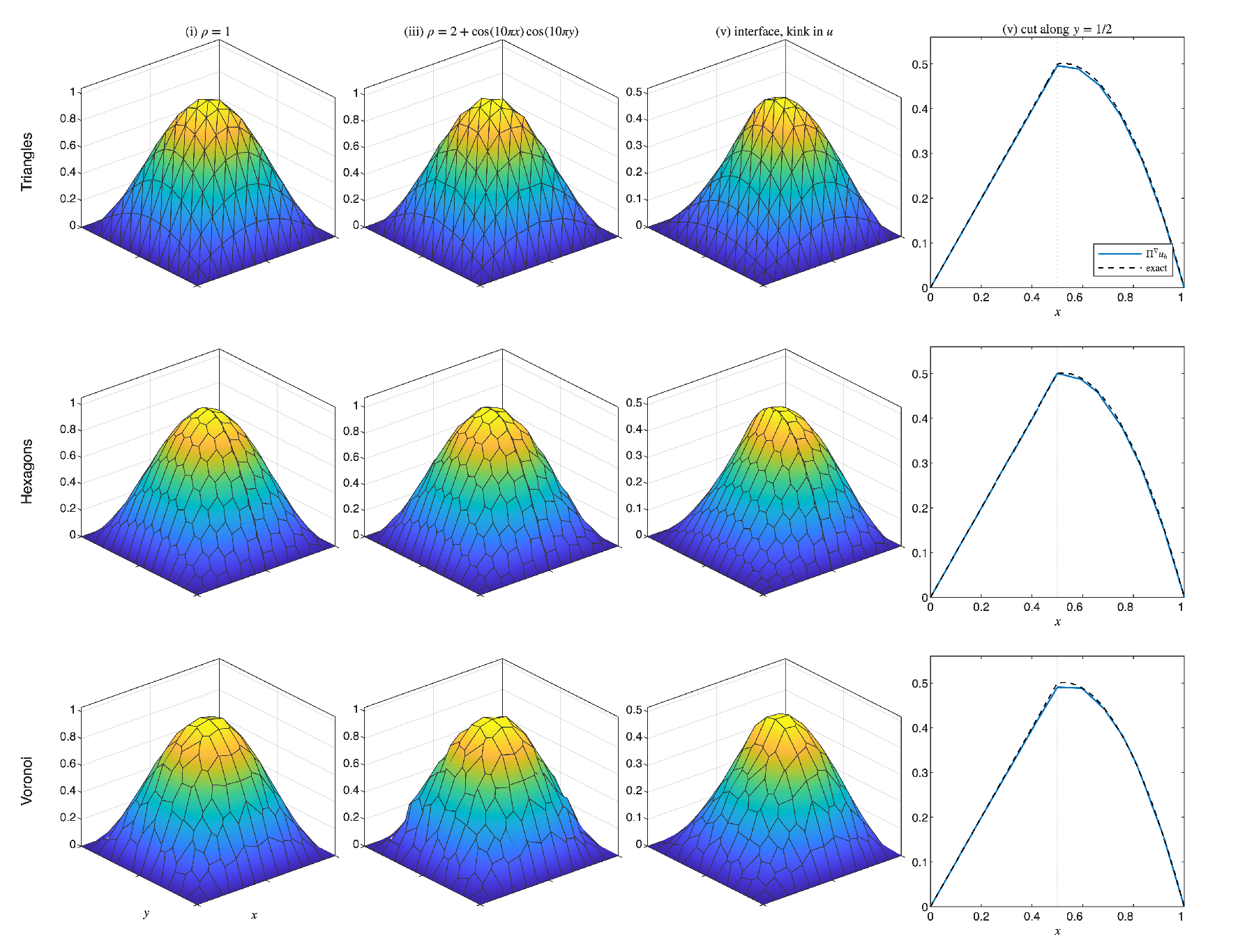}
    \caption{Discrete solution $u_h$ for the coefficients (i), (iii) and (v) of \cref{sec:convergence} on a mesh of each family: $288$ triangles with $h\approx 0.118$, $232$ hexagons with $h\approx 0.099$ and $144$ Voronoi cells with $h\approx 0.137$. Cases (ii) and (iv) are omitted since they share the exact solution of case (i). The rightmost column is the cut of case (v) along $y=1/2$.}
    \label{fig:solutions}
\end{figure}

\Cref{fig:convergence} shows the convergence of the method. Errors decay as $h^2$ in the $L^2$ norm and as $h$ in the broken $H^1$ seminorm, uniformly in the family of meshes and in the choice of the coefficient. 

\begin{figure}[h!]
    \centering
    \includegraphics[width=\linewidth]{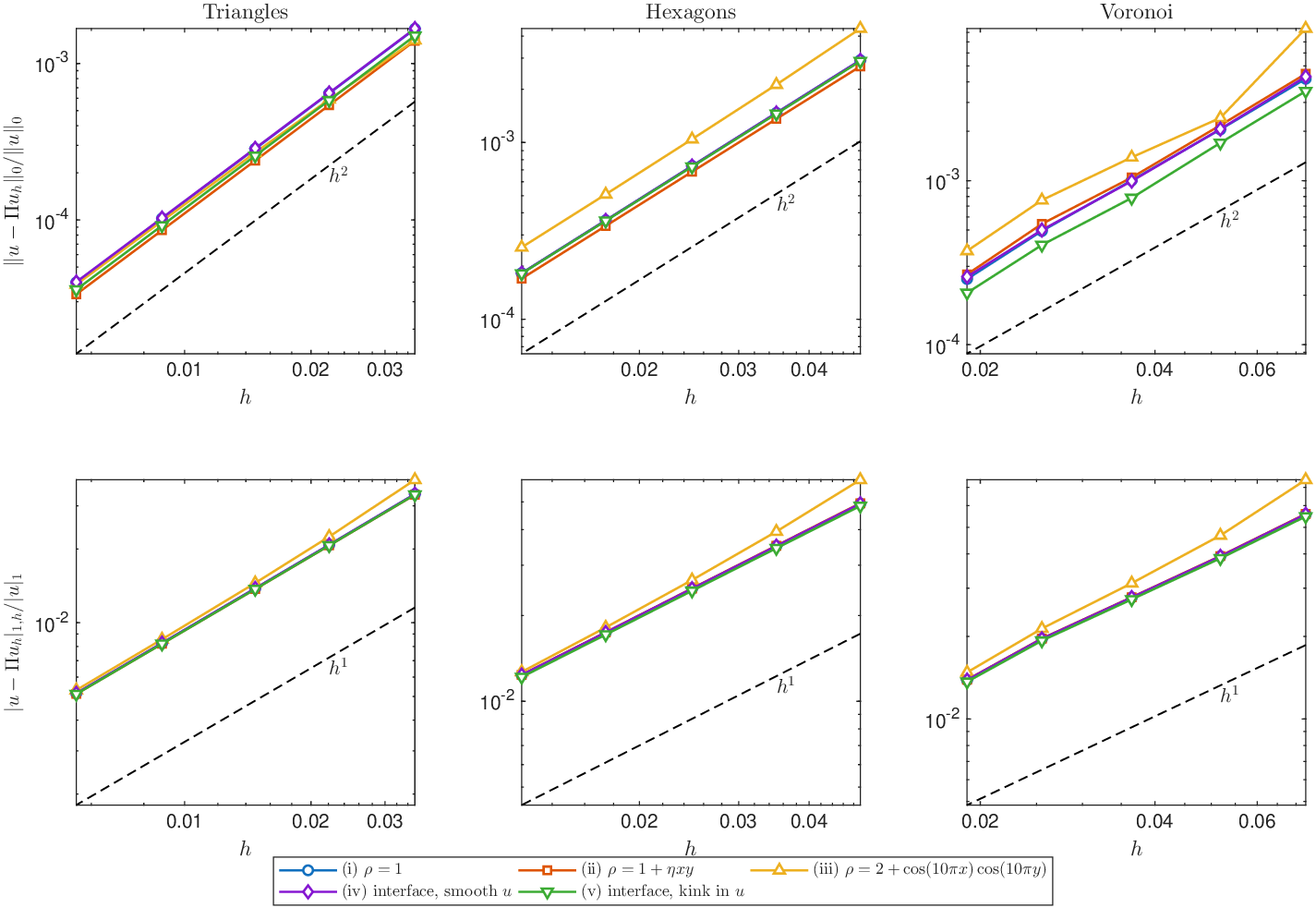}
    \caption{Relative errors in the $L^2$ norm (top) and in the broken $H^1$ seminorm (bottom) as a function of $h$, for the three families of meshes and the five coefficients of \cref{sec:convergence} with $\gamma_h=10$. The dashed lines have slopes $2$ and $1$.}
    \label{fig:convergence}
\end{figure}

\subsection{Preconditioner for the case of a constant coefficient}

We first set $\rho\equiv1$ to facilitate comparison with previous studies, such as \cite{Feng-Karakashian} for triangular meshes, and to assess the effect of different mesh geometries. \Cref{tab:my-table1} reports the results for structured triangular meshes as a function of the number of subdomains, while \cref{tab:my-table2} reports similar results for hexagonal and Voronoi meshes. 

\begin{table}[H]
\centering
\setlength{\tabcolsep}{3pt}
\begin{tabular}{|c|cc|ccc|ccc|}
\hline
& \multirow{2}{*}{$N$} & \multirow{2}{*}{$\dim V_0$} & \multicolumn{3}{c|}{$H/\delta = 16$} & \multicolumn{3}{c|}{$H/\delta = 4$} \\ \cline{4-9}
& & & $I_{\mathcal{H}}$ & $I_2$ & $I_4$ & $I_{\mathcal{H}}$ & $I_2$ & $I_4$ \\ \hline
\multirow{4}{*}{\rotatebox{90}{Square}}
& 64  & 49  & 20 (14.9) & 20 (13.5) & & 15 (6.2) & 15 (6.2) & \\ \cline{2-9}
& 144 & 121 & 20 (14.1) & 20 (13.5) & & 14 (6.0) & 15 (6.1) & \\ \cline{2-9}
& 256 & 225 & 20 (13.5) & 19 (14.4) & & 14 (5.8) & 15 (6.0) & \\ \cline{2-9}
& 400 & 361 & 18 (13.8) & 18 (15.0) & & 14 (5.7) & 15 (6.0) & \\ \hline
\multirow{4}{*}{\rotatebox{90}{METIS}}
& 64  & 97  & 37 (20.3) & 38 (20.4) & 37 (19.8) & 25 (9.3)  & 27 (11.0) & 25 (9.6)  \\ \cline{2-9}
& 144 & 243 & 40 (25.0) & 40 (23.5) & 40 (25.4) & 26 (9.8)  & 27 (10.3) & 26 (10.0) \\ \cline{2-9}
& 256 & 447 & 40 (24.8) & 40 (22.7) & 40 (23.5) & 26 (11.0) & 28 (12.5) & 27 (11.2) \\ \cline{2-9}
& 400 & 711 & 48 (38.6) & 48 (37.7) & 49 (38.3) & 29 (13.2) & 30 (14.2) & 29 (13.5) \\ \hline
\end{tabular}
\caption{Condition number ($\kappa$) for the preconditioned system $B_{ad}^{-1} B$ and number of PGC iterations ($I$) as a function of the number of square and irregular subdomains $N$ for triangular structured meshes with $\rho = 1$; $\dim V_0$ is the coarse space dimension and $H/h = 16$. The case $H/\delta = 16$ corresponds to the minimal overlap case $\delta = h$. The subscript records the extension used in $R_0^T$: $I_{\mathcal{H}}$ is the discrete harmonic extension, and $I_k$ the projector $\Pi^{\nabla}_{\Omega_i,k}$. Only $I_2$ is reported for the square subdomains: there the coarse elements are squares and the boundary data is piecewise linear, and both $k=2$ and $k=4$ recover the same coarse space.}
\label{tab:my-table1}
\end{table}

\begin{table}[H]
\centering
\setlength{\tabcolsep}{3pt}
\begin{tabular}{|c|cc|ccc|ccc|}
\hline
& \multirow{2}{*}{$N$} & \multirow{2}{*}{$\dim V_0$} & \multicolumn{3}{c|}{$H/\delta = 16$} & \multicolumn{3}{c|}{$H/\delta = 4$} \\ \cline{4-9}
& & & $I_{\mathcal{H}}$ & $I_2$ & $I_4$ & $I_{\mathcal{H}}$ & $I_2$ & $I_4$ \\ \hline
\multirow{4}{*}{\rotatebox{90}{Hexagonal}}
& 64  & 97  & 30 (18.3) & 31 (15.1) & 30 (16.0) & 21 (7.5) & 22 (7.9) & 21 (7.7) \\ \cline{2-9}
& 144 & 245 & 33 (18.7) & 32 (16.4) & 32 (16.8) & 21 (7.2) & 22 (7.2) & 21 (7.2) \\ \cline{2-9}
& 256 & 452 & 33 (20.4) & 33 (17.5) & 33 (18.5) & 22 (7.8) & 24 (8.8) & 23 (7.9) \\ \cline{2-9}
& 400 & 729 & 36 (24.8) & 36 (22.4) & 35 (22.9) & 23 (9.2) & 24 (9.4) & 24 (9.2) \\ \hline
\multirow{4}{*}{\rotatebox{90}{Voronoi}}
& 64  & 96  & 36 (24.1) & 39 (22.7) & 37 (20.5) & 27 (10.7) & 27 (10.5) & 27 (11.4) \\ \cline{2-9}
& 144 & 241 & 39 (26.7) & 39 (21.4) & 39 (22.6) & 29 (13.0) & 30 (13.8) & 29 (12.9) \\ \cline{2-9}
& 257 & 452 & 41 (28.1) & 41 (24.4) & 42 (25.5) & 30 (14.2) & 31 (13.8) & 31 (14.1) \\ \cline{2-9}
& 403 & 730 & 44 (30.6) & 44 (28.9) & 44 (28.4) & 30 (13.9) & 32 (14.3) & 31 (13.7) \\ \hline
\end{tabular}
\caption{Condition number ($\kappa$) for the preconditioned system $B_{ad}^{-1} B$ and number of PGC iterations ($I$) as a function of the number of irregular subdomains $N$ for hexagonal and Voronoi meshes with $\rho = 1$; $\dim V_0$ is the coarse space dimension and $H/h = 16$. The subscript records the extension used in $R_0^T$: $I_{\mathcal{H}}$ is the discrete harmonic extension, and $I_k$ the projector $\Pi^{\nabla}_{\Omega_i,k}$. 
}
\label{tab:my-table2}
\end{table}

\Cref{tab:discV0} reports the results obtained with the discontinuous reduced coarse space $V_0^{R,D}$ of \cref{rem:discreteCoarse}, for the same configuration as the last column of \cref{tab:my-table1} with no significant differences but with a coarse space three to four times larger. This comparison led us to concentrate on the continuous space $V_0^{R,C}$. We also show results for variations on $H/h$ and $H/\delta$; see \cref{tab:Hh_tria_rho0} and \cref{fig:enter-label}.

\begin{table}[H]
\centering
\begin{tabular}{|c|cc|cc|}
\hline
\multirow{2}{*}{\centering $N$} & \multicolumn{2}{c|}{Square subdomains} & \multicolumn{2}{c|}{METIS subdomains} \\ \cline{2-5} 
 & \multicolumn{1}{c|}{$\dim V_0$} & \multicolumn{1}{c|}{$I (\kappa)$} & \multicolumn{1}{c|}{$\dim V_0$} & \multicolumn{1}{c|}{$I (\kappa)$} \\ \hline
64 & \multicolumn{1}{c|}{196} & \multicolumn{1}{c|}{17 (5.9)} & \multicolumn{1}{c|}{291} & \multicolumn{1}{c|}{26 (9.9)} \\ \hline
144 & \multicolumn{1}{c|}{484} & \multicolumn{1}{c|}{16 (5.8)} & \multicolumn{1}{c|}{729} & \multicolumn{1}{c|}{27 (11.1)}  \\ \hline
256 & \multicolumn{1}{c|}{900} & \multicolumn{1}{c|}{15 (5.6)} & \multicolumn{1}{c|}{1341} & \multicolumn{1}{c|}{28 (12.0)}  \\ \hline
400 & \multicolumn{1}{c|}{1444} & \multicolumn{1}{c|}{15 (5.6)} &  \multicolumn{1}{c|}{2134} & \multicolumn{1}{c|}{31 (15.0)}  \\ \hline
\end{tabular}
\caption{Condition number ($\kappa$) for the preconditioned system $B_{ad}^{-1} B$ and number of PGC iterations ($I$) as a function of the number of square and irregular subdomains $N$ for triangular structured meshes with $\rho = 1$; $\dim V_0$ is the coarse space dimension, $H/h = 16$ and $H/\delta = 4$. In this case we used the discontinuous reduced coarse space $V_0^{R,D}$ of \cref{rem:discreteCoarse}; compare with \cref{tab:my-table1}. This comparison was carried out only with the discrete harmonic extension in $R_0^T$.}
\label{tab:discV0}
\end{table}

\begin{table}[H]
\centering
\setlength{\tabcolsep}{3pt}
\begin{tabular}{|c|ccc|cccc|}
\hline
\multirow{2}{*}{$H/h$} & \multicolumn{3}{c|}{Square subdomains} & \multicolumn{4}{c|}{METIS subdomains} \\ \cline{2-8}
 & $\dim V_0$ & $I_{\mathcal{H}}$ & $I_2$ & $\dim V_0$ & $I_{\mathcal{H}}$ & $I_2$ & $I_4$ \\ \hline
12 & 49 & 15 (6.05) & 15 (5.95) & 97 & 27 (10.70) & 27 (10.94) & 27 (10.50) \\ \hline
16 & 49 & 15 (6.24) & 15 (6.23) & 97 & 25 (9.28)  & 27 (10.95) & 25 (9.56)  \\ \hline
20 & 49 & 15 (6.30) & 15 (6.29) & 93 & 24 (8.73)  & 26 (10.02) & 25 (9.04)  \\ \hline
24 & 49 & 15 (6.35) & 15 (6.34) & 97 & 27 (12.58) & 29 (14.71) & 28 (12.91) \\ \hline
28 & 49 & 15 (6.39) & 15 (6.38) & 96 & 26 (9.84)  & 29 (12.05) & 27 (10.22) \\ \hline
32 & 49 & 15 (6.43) & 16 (6.42) & 95 & 26 (10.56) & 30 (13.08) & 27 (10.99) \\ \hline
\end{tabular}
\caption{Condition number ($\kappa$) for the preconditioned system $B_{ad}^{-1} B$ and number of PGC iterations ($I$) as a function of $H/h$, for triangular structured meshes with $\rho = 1$; $\dim V_0$ is the coarse space dimension, $N=64$ and $H/\delta = 4$. The subscript records the extension used in $R_0^T$: $I_{\mathcal{H}}$ is the discrete harmonic extension, and $I_k$ the projector $\Pi^{\nabla}_{\Omega_i,k}$. Only $I_2$ is reported for the square subdomains, as in \cref{tab:my-table1}.}
\label{tab:Hh_tria_rho0}
\end{table}

\begin{figure}[H]
    \centering
    \includegraphics[width=0.48\linewidth]{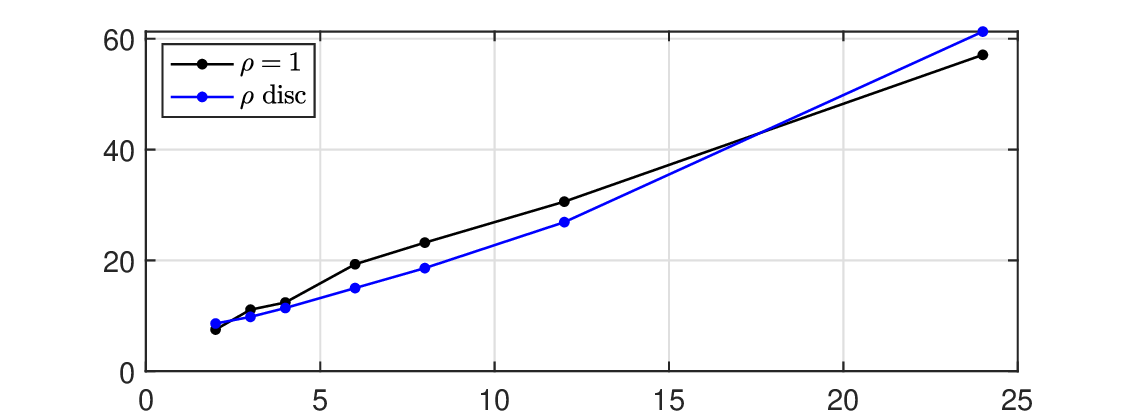}
    \includegraphics[width=0.48\linewidth]{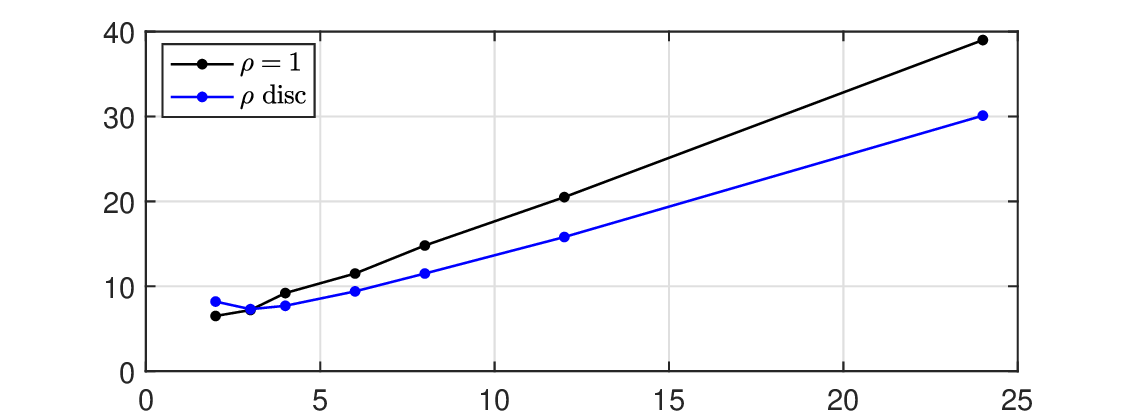}
    \caption{Dependence of the condition number of the preconditioned system for (left) triangular and (right) hexagonal meshes, with irregular subdomains. We use $H/h=24$, $N=64$. }
    \label{fig:enter-label}
\end{figure}

\subsection{Piecewise constant coefficient}\label{sec:rhodisc}
We now let $\rho$ be constant on each subdomain, with its value sampled independently from a log-uniform distribution on $[1,10^6]$. Thus, up to sampling fluctuations, the coefficient contrast is $10^6$. The results are reported in \cref{tab:my-table3} for triangular meshes, in \cref{tab:my-table4} for hexagonal and Voronoi meshes, and in \cref{tab:Hh_tria_rho1} as a function of $H/h$.

For the METIS partitions, as well as for the hexagonal and Voronoi meshes, the condition numbers remain of the same order as those obtained with $\rho\equiv1$ and are essentially independent of $N$, in agreement with \cref{main_thm}. In contrast, for the square partitions of the triangular meshes, the condition numbers increase steadily with the number of subdomains. This behavior can be attributed to the failure of the quasi-monotonicity condition on these partitions. Indeed, at an interior vertex of a structured partition, four subdomains meet, and Assumption \ref{ass:monotonoCoef} fails whenever coefficient has a configuration as described in \cref{rem:qmSharp}. On the other hand, the dependence on $H/h$ shown in \cref{tab:Hh_tria_rho1} remains essentially flat for both types of partition, as was also observed for $\rho\equiv1$.

\begin{table}[h!]
\centering
\setlength{\tabcolsep}{3pt}
\begin{tabular}{|c|cc|ccc|ccc|}
\hline
& \multirow{2}{*}{$N$} & \multirow{2}{*}{$\dim V_0$} & \multicolumn{3}{c|}{$H/\delta = 16$} & \multicolumn{3}{c|}{$H/\delta = 4$} \\ \cline{4-9}
& & & $I_{\mathcal{H}}$ & $I_2$ & $I_4$ & $I_{\mathcal{H}}$ & $I_2$ & $I_4$ \\ \hline
\multirow{4}{*}{\rotatebox{90}{Square}}
& 64  & 49  & 33 (22.2)  & 33 (22.2)  & & 26 (19.5)  & 26 (19.5)  & \\ \cline{2-9}
& 144 & 121 & 52 (57.6)  & 52 (57.6)  & & 46 (50.9)  & 46 (50.9)  & \\ \cline{2-9}
& 256 & 225 & 85 (270.4) & 85 (270.5) & & 73 (210.6) & 73 (210.8) & \\ \cline{2-9}
& 400 & 361 & 107 (432.5)& 107 (432.6)& & 97 (351.7) & 97 (351.8) & \\ \hline
\multirow{4}{*}{\rotatebox{90}{METIS}}
& 64  & 97  & 36 (18.4) & 36 (17.9) & 36 (18.9) & 26 (8.1)  & 28 (9.7)  & 26 (8.2)  \\ \cline{2-9}
& 144 & 243 & 37 (17.9) & 38 (17.8) & 37 (17.8) & 28 (10.8) & 31 (13.1) & 29 (11.2) \\ \cline{2-9}
& 257 & 447 & 45 (27.4) & 47 (27.0) & 46 (27.5) & 31 (10.9) & 35 (13.2) & 31 (11.6) \\ \cline{2-9}
& 400 & 711 & 49 (25.7) & 51 (26.0) & 49 (25.8) & 30 (11.0) & 34 (12.7) & 31 (11.0) \\ \hline
\end{tabular}
\caption{Condition number ($\kappa$) for the preconditioned system $B_{ad}^{-1} B$ and number of PGC iterations ($I$) as a function of the number of square and irregular subdomains $N$ for triangular structured meshes with the piecewise constant coefficient of \cref{sec:rhodisc}; $\dim V_0$ is the coarse space dimension and $H/h = 16$. The case $H/\delta = 16$ corresponds to the minimal overlap case $\delta = h$. The subscript records the extension used in $R_0^T$: $I_{\mathcal{H}}$ is the discrete harmonic extension, and $I_k$ the projector $\Pi^{\nabla}_{\Omega_i,k}$.}
\label{tab:my-table3}
\end{table}

\begin{table}[h!]
\centering
\setlength{\tabcolsep}{3pt}
\begin{tabular}{|c|cc|ccc|ccc|}
\hline
& \multirow{2}{*}{$N$} & \multirow{2}{*}{$\dim V_0$} & \multicolumn{3}{c|}{$H/\delta = 16$} & \multicolumn{3}{c|}{$H/\delta = 4$} \\ \cline{4-9}
& & & $I_{\mathcal{H}}$ & $I_2$ & $I_4$ & $I_{\mathcal{H}}$ & $I_2$ & $I_4$ \\ \hline
\multirow{4}{*}{\rotatebox{90}{Hexagonal}}
& 64  & 97  & 30 (15.0) & 31 (14.9) & 30 (15.0) & 23 (7.6) & 23 (7.6) & 24 (7.6) \\ \cline{2-9}
& 144 & 245 & 37 (16.7) & 37 (16.5) & 37 (16.5) & 25 (7.4) & 25 (7.3) & 25 (7.3) \\ \cline{2-9}
& 256 & 452 & 41 (20.6) & 42 (20.0) & 41 (20.5) & 27 (9.1) & 28 (9.4) & 27 (9.1) \\ \cline{2-9}
& 400 & 729 & 39 (20.7) & 40 (20.3) & 39 (20.4) & 27 (8.5) & 29 (9.3) & 27 (8.3) \\ \hline
\multirow{4}{*}{\rotatebox{90}{Voronoi}}
& 64  & 96  & 40 (20.1) & 41 (20.5) & 41 (20.4) & 29 (9.9)  & 30 (11.2) & 29 (10.0) \\ \cline{2-9}
& 144 & 241 & 43 (21.5) & 45 (22.8) & 43 (21.9) & 31 (12.6) & 34 (13.6) & 32 (12.9) \\ \cline{2-9}
& 257 & 452 & 46 (21.8) & 46 (22.1) & 45 (21.9) & 34 (13.4) & 37 (18.4) & 34 (13.5) \\ \cline{2-9}
& 403 & 730 & 47 (28.1) & 51 (29.5) & 48 (28.5) & 36 (14.9) & 45 (23.5) & 38 (15.8) \\ \hline
\end{tabular}
\caption{Condition number ($\kappa$) for the preconditioned system $B_{ad}^{-1} B$ and number of PGC iterations ($I$) as a function of the number of irregular subdomains $N$ for hexagonal and Voronoi meshes with the piecewise constant coefficient of \cref{sec:rhodisc}; $\dim V_0$ is the coarse space dimension and $H/h = 16$. The subscript records the extension used in $R_0^T$, as in \cref{tab:my-table3}.}
\label{tab:my-table4}
\end{table}

\begin{table}[h!]
\centering
\setlength{\tabcolsep}{3pt}
\begin{tabular}{|c|ccc|cccc|}
\hline
\multirow{2}{*}{$H/h$} & \multicolumn{3}{c|}{Square subdomains} & \multicolumn{4}{c|}{METIS subdomains} \\ \cline{2-8}
 & $\dim V_0$ & $I_{\mathcal{H}}$ & $I_2$ & $\dim V_0$ & $I_{\mathcal{H}}$ & $I_2$ & $I_4$ \\ \hline
12 & 49 & 26 (19.6) & 26 (19.7) & 97 & 26 (9.6)  & 29 (11.4) & 27 (9.8)  \\ \hline
16 & 49 & 26 (19.5) & 26 (19.5) & 97 & 26 (8.1)  & 28 (9.7)  & 26 (8.2)  \\ \hline
20 & 49 & 26 (19.5) & 26 (19.5) & 93 & 25 (7.9)  & 28 (9.5)  & 25 (8.0)  \\ \hline
24 & 49 & 26 (19.4) & 26 (19.4) & 97 & 26 (11.9) & 30 (13.6) & 27 (12.3) \\ \hline
28 & 49 & 26 (19.3) & 26 (19.3) & 96 & 28 (9.8)  & 31 (11.7) & 29 (10.0) \\ \hline
32 & 49 & 26 (19.3) & 26 (19.3) & 95 & 28 (9.9)  & 32 (11.4) & 29 (10.2) \\ \hline
\end{tabular}
\caption{Condition number ($\kappa$) for the preconditioned system $B_{ad}^{-1} B$ and number of PGC iterations ($I$) as a function of $H/h$, for triangular structured meshes with the piecewise constant coefficient of \cref{sec:rhodisc}; $\dim V_0$ is the coarse space dimension, $N=64$ and $H/\delta = 4$. The subscript records the extension used in $R_0^T$, as in \cref{tab:my-table3}.}
\label{tab:Hh_tria_rho1}
\end{table}

\section{Conclusions and final remarks} \label{section:conclusions}
We have introduced and analyzed a two-level overlapping additive Schwarz preconditioner for discontinuous Galerkin discretizations of virtual element spaces on general polygonal meshes. The coarse space is built from a virtual, continuous partition of unity with a single degree of freedom per subdomain vertex. The bound obtained in \cref{main_thm} is valid with a constant $C$ independent of $H$, $h$, $\delta$, of the number of subdomains and of the contrast of the coefficient. This is the same bound that is known for the conforming virtual element case \cite{CalvoJ}, so that the discontinuous discretization does not degrade the performance of the preconditioner.

Two remarks are worth making. First, robustness with respect to the contrast is obtained by using $\rho$-weighted averages twice: in the discrete formulation, where they also restore the consistency of the method, and in the construction of the continuous function $\chi$ of \cref{lemmachi}, which is the technical tool that transfers estimates from the discontinuous space to the conforming one. Second, the whole role of the quasi-monotonicity \textcolor{red}{\cref{ass:monotonoCoef}} is confined to a single spectral estimate. \textcolor{red}{\cref{N_term2}} shows that the constant relating the weighted variance of the nodal values around a node to the weighted jumps across the edges of its patch is, for every patch and every coefficient, the reciprocal of the spectral gap $\lambda_2(L_\rho,D)$ of a local generalized eigenvalue problem. Thus, \textcolor{red}{\cref{ass:monotonoCoef}} is a sufficient condition for that gap to be bounded below by $1/m^2$, independently of the coefficient. The assumption cannot be removed: for the checkerboard configuration of \cref{rem:qmSharp} shows that the optimal constant equals exactly $(M+1)/4$, where $M$ is the contrast.

Some questions remain open. The discontinuous reduced coarse space $V_0^{R,D}$ of \cref{rem:discreteCoarse} is not covered by our analysis, since the coarse component of the stable decomposition relies on the continuity of $R_0^Tu_0$; our single numerical comparison suggests that it is not worth pursuing, but a proof of a comparable bound would be of interest. Finally, the extension of the construction to three dimensions and to higher-order discontinuous virtual element spaces, is left for future work.

\section*{Acknowledgments}
J.G.C. gratefully acknowledges the institutional support for project C5054 provided by the \textit{Vicerrectoría de Investigación, Universidad de Costa Rica}, through the Research Center in Pure and Applied Mathematics (CIMPA) and the School of Mathematics.

\section*{Declaration of AI-Assisted Tools}
The authors declare that Claude Opus 5 was used exclusively for grammatical proofreading, transcription of handwritten mathematical calculations, and stylistic refinement. The discretization, preconditioner design, and theoretical analysis are original contributions of the authors. The authors assume full responsibility for the accuracy, integrity, and scientific content of the manuscript.

\bibliographystyle{siamplain}
\bibliography{references}
\end{document}